\documentclass[preprint,12pt,authoryear]{elsarticle}

\usepackage{subcaption}
\usepackage{amssymb}
\usepackage{amsmath}
\usepackage{float}
\usepackage{xcolor}
\usepackage{graphicx}
\usepackage{url}
\numberwithin{equation}{section}
\usepackage{amsthm}
\usepackage{caption}
\usepackage{ulem}
\usepackage{multirow}
\usepackage{array} 
\usepackage{booktabs}
\newtheorem{prop}{Proposition}

\newtheorem{defi}{Definition}
\newtheorem{remark}{Remark}
\newcommand{\R}{\mathbb{R}}
\newcommand{\Esp}{\mathbb{E}}
\newcommand{\Cov}{\mathbb{C}\mathrm{ov}}
\newcommand{\Var}{\mathbb{V}\mathrm{ar}}

\newcommand{\coef}{c}
\newcommand{\basis}{v}
\newcommand{\basisVec}{\basis_{., \ell}}
\newcommand{\nbasis}{m}
\newcommand{\nvar}{\ensuremath{d}}
\newcommand{\dimInt}{\{1, \dots, d\}}
\newcommand{\sample}[2]{#1^{(#2)}}
\newcommand{\sobolClosed}[1]{S_{#1}^{\mathrm{c}}}
\newcommand{\sobolClosedD}[1]{D_{#1}^{\mathrm{c}}}
\newcommand{\sobolTot}[1]{S_{#1}^{\mathrm{tot}}}
\newcommand{\sobolTotD}[1]{D_{#1}^{\mathrm{tot}}}
\newcommand{\pfEstim}[1]
{\widehat{#1}^{\mathrm{pf}}}

\newcommand{\spatialLocation}{\ensuremath{z}} 
\newcommand{\inputVector}{\ensuremath{X} }

\newcommand{\nOutputDimensions}{\ensuremath{L}}

\newcommand{\numberPickFreezeSamples}{\ensuremath{N}}
\newcommand{\probDist}{\mu}
\newcommand{\indexInput}{i}
\newcommand{\nInputDimensions}{d}
\newcommand{\outputFunction}{Y}
\newcommand{\indexSet}{I}
\newcommand{\SobolNormalized}{S}
\newcommand{\SobolUnnormalized}{D}
\newcommand{\exceptIndexSet}{\sim \indexSet}
\newcommand{\outputValues}{y_{\ell}}
\newcommand{\costDW}{\mathrm{Cost}_\text{DW}}
\newcommand{\costBD}{\mathrm{Cost}_\text{BD}}

\begin{document}

\begin{frontmatter}

\title{Fast pick-freeze estimation of Sobol' sensitivity maps using basis expansions}

\author[label1,label2]{Yuri Taglieri Sáo} 

\author[label1]{Olivier Roustant}

\affiliation[label1]{organization={Institut de Mathématiques de Toulouse, Université de Toulouse, INSA},
            addressline={135, Avenue de Rangueil}, 
            city={Toulouse},
            postcode={31077}, 
            state={Occitanie},
            country={France}}

\author[label2]{Geraldo de Freitas Maciel}

\affiliation[label2]{organization={Engineering College of Ilha Solteira, Civil Engineering Department, São Paulo State University ``Júlio de Mesquita Filho'' (UNESP)},
            addressline={Av. Brasil, 56}, 
            city={Ilha Solteira},
            postcode={15385-000}, 
            state={São Paulo},
            country={Brazil}}

\begin{abstract}

Global sensitivity analysis (GSA) aims at quantifying the contribution of input variables over the variability of model outputs. In the frame of functional outputs, a common goal is to compute \textit{sensitivity maps} (SM), i.e sensitivity indices at each output dimension (e.g. time step for time series, or pixels for spatial outputs). In specific settings, some works have shown that the computation of Sobol' SM can be speeded up by using basis expansions employed for dimension reduction. However, how to efficiently compute such SM in a general setting has not received too much attention in the GSA literature.\\
In this work, we propose fast computations of Sobol' SM using a general basis expansion, with a focus on statistical estimation. First, we write a closed-form expression of SM in function of the matrix-valued Sobol' index of the vector of basis coefficients. Secondly, we consider pick-freeze (PF) estimators, which have nice statistical properties (in terms of asymptotical efficiency) for Sobol' indices of any order. We provide similar \textit{basis-derived} formulas for the PF estimator of Sobol' SM in function of the matrix-valued PF estimator of the vector of basis coefficients. We give the computational cost, and show that, compared to a \textit{dimension-wise} approach, the computational gain is substantial and allows to calculate both SM and their associated bootstrap confidence bounds in a reasonable time. Finally, we illustrate the whole methodology on an analytical test case and on an application in non-Newtonian hydraulics, modelling an idealized dam-break flow. 
\end{abstract}

\begin{highlights}
\item Fast computations of Sobol' sensitivity maps (SMs) based on a general basis expansion;
\item Closed-form expression of SMs using pick-freeze formulas directly in the basis coefficients (\textit{basis-derived} approach);
\item Computational gain analysis of the \textit{basis-derived} approach over the \textit{dimension-wise} approach;
\item Application to an analytical test-case (Campbell2D function) and to a non-Newtonian fluid mechanics data-driven case.
\end{highlights}


\end{frontmatter}

\section{Introduction}
\label{Sec.introduction}

Global Sensitivity Analysis (GSA) consists of numerous techniques that apportion the uncertainty of a model output to uncertainties in its inputs \citep{Saltelli2008}. We can understand as ``model'' a computational code that receives input variables (or \textit{factor} or \textit{parameter}) and produces outputs, which may be scalars or vectors, that correspond to real phenomena. GSA has remarkable importance in any study that requires knowing which input variables contributes the most to the variability of outputs, such as in model validation, model calibration and decision-making processes \citep{Morris1991, Sobol2001, Saltelli2008}. In general, but not limited to, GSA aims to identify non-influential and important inputs, to rank the importance of such parameters and to understand the model based on the interactions of the parameters \citep{DaVeiga2021, Iooss2015}. Therefore, GSA stands itself as an invaluable tool for studying any uncertainty-related subject.

An important group of GSA techniques relies on quantifying the variance as uncertainty measure, the so-called ``variance-based methods''. Since the variance quantifies the dispersion of data around the mean, it represents itself as an adequate metric to evaluate the uncertainty of input variables over outputs \citep{Saltelli2010, Borgonovo2016}. The variance also presents convenient statistical properties, which allow us to perform a variance-decomposition procedure (Sobol-Hoeffding decomposition) and obtain the Sobol' indices, which are intuitive sensitivity indices capable of measuring the influence of input variables and their interaction over the variability of the outputs \citep{Sobol2001, Saltelli2008}. 

However, closed analytical formulas for calculating the Sobol' indices exist just for some classes of models, such as Polynomial Chaos Expansion (PCE) \citep{Sudret2008}, and some simple and explicit models, generally used as benchmark cases \citep{Becker2020}. In the majority of practical cases, the Sobol' indices are estimated by sampling-based estimators, which consists of obtaining a finite sample of model evaluations, typically with size on the order of thousands, and applying a discrete estimation technique \citep{Helton2006, Saltelli2008}. This kind of technique is called a ``pick-freeze'' (PF) scheme \citep{Sobol1993, Saltelli2002, Gamboa2016}. Its main advantage is that it is a very general scheme, whose only assumption is the square-integrability of the model with respect to the probability distributions of input variables \citep{DaVeiga2021}. Some works were dedicated to improve the accuracy and efficiency of PF schemes by adopting different sampling strategies \citep{Devroye1986, McKay1995, Tissot2015}; others studied their asymptotic properties for purely Monte-Carlo sampling \citep{Janon2014, Gamboa2016}.

Although very general, PF schemes were primarily developed to perform GSA of scalar outputs. However, many practical cases study temporal and/or spatial outputs, i.e. functional outputs, high-dimensional by definition. This high-dimensionality characteristic means that performing GSA produces results in the form of series or  maps of Sobol' indices. Herein, we call this series/map as \textit{Sensitivity Map} (SM). Therefore, the PF scheme may not be readily applicable, requiring some type of pre-GSA approach to deal with the high dimensionality of outputs. The most direct approach consists in discretizing the domain with a finite number of coordinates, where each coordinate produces a scalar value \citep{Terraz2017}. Then, the model is evaluated on the order of thousand times and the PF scheme is performed on each coordinate, ultimately producing SMs. Although very informative and capable of providing sensitivity indices maps to analysts, two problems may arise from this approach: 1) computational storage may be prohibitive; and 2) the model can be computationally expensive to produce enough realizations for an accurate Sobol' indices estimation through PF schemes. The first problem can be addressed by using dimension-reduction techniques, such as basis expansions (e.g. Principal Component Analysis, B-splines, wavelets), to reduce output dimensionality, consequently reducing memory storage \citep{Abdi2010, Hawchar2017, Nagel2020}. The second problem requires the construction of a metamodel (also called surrogate model or emulator), which approximates accurately the model and is fast to evaluate. Famous families of metamodels include Gaussian process regression \citep{Williams1995}, polynomial chaos expansion \citep{Sudret2008}, artificial neural networks \citep{Zou2009, Fonseca2003}. 

In this context, by applying a basis-expansion technique and by metamodeling and predicting the basis coefficients through the metamodels, the original high-dimensional outputs can be decoded from the reduced-space variables, i.e. the basis components and coefficients, using a scalar product operation. \cite{Marrel2011} chose Principal Component Analysis (PCA) as a basis-expansion technique and used Gaussian Process Regression (GPR) to metamodel the PCA coefficients and to perform GSA using the PF scheme directly in each output dimension, obtaining the SMs. In this work, we name this approach as \textit{dimension-wise} approach, since the PF scheme is performed in each output dimension separately. However, this methodology overlooks the fact that basis coefficients and components could be used for a more efficient GSA framework. In this direction, \cite{Nagel2020} also chose PCA, but used Polynomial Chaos Expansion (PCE) to metamodel the PCA coefficients. By leveraging the orthogonality properties of PCE, the Sobol' indices of the PCA coefficients are calculated analytically and, therefore, no PF scheme is needed. By decoding the Sobol' indices from the reduced space to the original space, SMs were obtained. Herein, we name this approach as \textit{basis-derived} approach, since the SMs are estimated directly from the basis coefficients and components. Although efficient, it is not general: the analytical treatment of Sobol' indices makes the framework limited only to PCE. The application of other metamodeling techniques, e.g. ANN (Artificial Neural Networks) and Gaussian Process Regression (GPR), would not be possible since they require a closed expression for the Sobol' indices. In a similar framework as \cite{Nagel2020}, using PCA and GPR metamodeling, \cite{Li2020} compute the Sobol' indices by using the basis coefficients. Eventually, notice that the aforementioned works do not explore the statistical estimation of the Sobol' indices computed with basis coefficients.

In this work, we investigate the \textit{basis-derived} approach to obtain SMs for general basis expansions. First, we obtain a closed-form expression of SM in function of the matrix-valued Sobol’ index of the vector of basis coefficients (Prop. \ref{prop:sobolWithBasis}). Then, we focus on statistical estimation. We show that the same formula remains true when we replace the Sobol' indices by their PF estimators (Prop. \ref{prop:sobolHatWithBasis}). This remarkable property is due to the fact that usual PF estimators are quadratic form of the two samples used in PF schemes (Remark \ref{rem:pickFreezeFormulaAndQuadraticForm}). We assess the computational cost (or numerical complexity) of the \textit{basis-derived} and \textit{dimension-wise} approaches (Prop. \ref{prop:ratio_costs}). It appears that the \textit{basis-derived} approach is much less computationally demanding and allows to estimate both the SMs and their bootstrap confidence bounds in a reasonable time. We illustrate the performance of the \textit{basis-derived} approach on an analytical test-case (Campbell2D function), which has spatial maps as outputs, and on an idealized gradual dam-break flow of non-Newtonian fluid flow, which produces Sobol' indices time series. 

The work is organized as follows. Section \ref{Sec.background} presents theoretical background and concepts of GSA. Section \ref{Sec.contributions} presents our contribution to the study of PF schemes. Section \ref{Sec.applications} shows the applications chosen for this work, i.e. the Campbell2D function and the gradual dam-break problem of non-Newtonian fluids. Finally, Section \ref{sec.conclusion} closes with the conclusions and future works.

\section{Background}
\label{Sec.background}

In this section, we present the theoretical ground of the variance decomposition and the concepts of sensitivity indices for variance-based GSA. Except for Section \ref{sec:PickFreezeVectorValued}, functions are scalar-valued. Most of the material is standard, and can be found e.g. in \cite{DaVeiga2021}.

\subsection{Sobol-Hoeffding decomposition and Sobol' indices for (scalar-valued case)}

We represent $\inputVector$ as the vector of input variables $\inputVector=(\inputVector_1, ..., \inputVector_{\nInputDimensions}) \in \mathbb{R}^{\nInputDimensions} $ and $\nInputDimensions$ the number of input variables. Each input variable has uncertainties associated to it and we model them as random variables, denoting $\probDist_{\inputVector}$ as the probability distribution of $\inputVector$. The components $\inputVector_{\indexInput}$ are assumed to be independent, probability distribution $\probDist_{\inputVector_{\indexInput}}$. The corresponding model output, here assumed to be scalar, is a random variable written as

\begin{equation}
    \outputFunction = f(\inputVector).
\end{equation}

We assume that $E(Y^2)<+\infty$. Equivalently, $f$ belongs to the Hilbert space $\mathbb{L}^2(\probDist_{\inputVector})$. We denote by $\langle ., . \rangle$ its inner product : $\langle f, g \rangle = \int fg \, d\mu_X$. According to the Sobol-Hoeffding decomposition \citep[see e.g.][]{Sobol1993}, $f(\inputVector)$ can be uniquely decomposed as a sum of terms with increasing complexity (Eq. \ref{Eq.hoeffding_decomposition}).

\begin{equation}
    f(\inputVector) = \sum_{\indexSet \in \mathcal{P}_{\nInputDimensions}} f_I(\inputVector_{\indexSet})
    \label{Eq.hoeffding_decomposition}
\end{equation}

such that for all $I \in \mathcal{P}_{\nInputDimensions}$ and all strict subset $J \subsetneq I$, $\Esp[f_I(X_I) \vert X_J] = 0$ (with the convention $\Esp[. \vert X_\emptyset] = \Esp[.]$).

Two properties hold for the terms in Eq. \ref{Eq.hoeffding_decomposition}. First, the zero-mean property or centered-mean property, where $\Esp [ f_{\indexSet} (\inputVector_{\indexSet}) ] = 0 $. Second, the orthogonality property in $\mathbb{L}^2(\probDist_{\inputVector})$, i.e. $\langle f_{\indexSet}, f_{\indexSet^{'}} \rangle = \Esp[ f_\indexSet(\inputVector_{\indexSet}) f_{\indexSet^{'}}(\inputVector_{\indexSet^{'}}) ] = 0$. This leads to the variance decomposition formula (Eq. \ref{Eq.decomposed_variance}):

\begin{equation}
    \Var(f(\inputVector)) = \sum_{\indexSet \in \mathcal{P}_{\nInputDimensions}, \indexSet \neq \emptyset} \Var(f_\indexSet(\inputVector_{\indexSet}))
    \label{Eq.decomposed_variance}
\end{equation}

Equation \ref{Eq.decomposed_variance} allows us to account quantitatively the sensitivity of each input and their interaction over the outputs. The influence of the group of input variables $\inputVector_{\indexSet}$ can be quantified by the variance explained by $f_{\indexSet}(\inputVector_{\indexSet})$. We call this quantity \textit{Sobol' index}: if normalized by the total variance, we represent it as $\SobolNormalized_{(\cdot)}$; otherwise, we represent the unnormalized index (also called partial variance) as $\SobolUnnormalized_{(\cdot)}$. The overall variance is represented by $\SobolUnnormalized$, without any subscript. Definition \ref{Def.Sobol_indices_scalar} provides the expressions of Sobol' indices and the different types of indices (first-order, closed, total and second-order). A brief interpretation follows.

\begin{defi}[(Unnormalized) Sobol' indices] 
Let $\indexSet \subseteq \dimInt$ and subscript ${(\cdot)}_{\exceptIndexSet}$ to indicate that variables from the subset $\indexSet$ were excluded from the whole set. The following equations are associated to the group variables $\inputVector_{\indexSet}$.
\label{Def.Sobol_indices_scalar}
\begin{itemize}
    \item First-order Sobol' index: $\SobolUnnormalized_{\indexSet} = \Var \left( f_{\indexSet}(\inputVector_{\indexSet}) \right)$.
    \item Closed Sobol' index: $\SobolUnnormalized_{\indexSet}^{c} = \Var \left( \Esp \left[ f(\inputVector) \vert \inputVector_{\indexSet} \right] \right) $.
    \item Total Sobol' index (see Remark \ref{Remark.total_Sobol}): $\SobolUnnormalized_{\indexSet}^{tot} = \SobolUnnormalized - \SobolUnnormalized_{\exceptIndexSet}^{c}$.
    \item Second-order Sobol' index (see Remark \ref{Remark.2nd_Sobol}): $\SobolUnnormalized_{\{i,j \}} = D_{\{i,j\}}^{c} - D_i - D_j$, where $i,j \in \dimInt$. 
\end{itemize}

\end{defi}

For simplicity, let us choose $I = \{ i \}$,  $I^{'} = \{ i, j \}$ and consider unnormalized indices. The first-order Sobol' index of input variable $X_i$, given by $\SobolUnnormalized_i$, is the isolated contribution of $X_i$, or main effect of $X_i$. The closed Sobol' index of $X_{I^{'}}$, given by $\sobolClosedD{i,j}$, contains the isolated effects of $X_i$ and $X_j$ and the interaction between both variables; therefore $\sobolClosedD{i,j} = \SobolUnnormalized_{i} + \SobolUnnormalized_{j} + \SobolUnnormalized_{ij}$. If $I^{'} = I = {\{ i \}}$, we have $\sobolClosedD{i} = \SobolUnnormalized_i$. The total Sobol' index of input variables $X_i$ and $X_j$, given by $\SobolUnnormalized^{tot}_{I^{'}}$, represents the main effect of both variables and the interaction of $\{ X_i, X_j\}$ with the remaining ones $X_{\sim I^{'}}$. Lastly, the second-order Sobol' index of input variables $X_i$ and $X_j$, given by $\SobolUnnormalized_{i,j}$, is the interaction between both of them.

\begin{remark}
\label{Remark.total_Sobol}
The total Sobol' index is related to the law of total variance:
$$\Var \left( f(\inputVector) \right) = \Var \left( \Esp \left[ f(\inputVector) \vert \inputVector_{\exceptIndexSet} \right] \right) + \Esp \left[ \Var \left( f(\inputVector) \vert \inputVector_{\exceptIndexSet} \right) \right]$$
which gives another interpretation of it as $\SobolUnnormalized_{\indexSet}^{tot} = \Esp \left[ \Var \left( f(\inputVector) \vert \inputVector_{\exceptIndexSet} \right) \right]$.
\end{remark}

\begin{remark}
\label{Remark.2nd_Sobol}
For all set $I$, each term of the Sobol-Hoeffding decomposition is given by the recursion formula $$ f_{\indexSet}(\inputVector_{\indexSet}) = \Esp[ f(\inputVector) \vert \inputVector_{\indexSet} ] - \sum_{J \subsetneq \indexSet} f_J(X_J)$$
This can be used to derive all Sobol' indices recursively from the closed Sobol' indices associated to subsets of variables. For instance, using the orthogonality of the decomposition, the second-order Sobol' index is written 

\begin{eqnarray*}
  \Var \left( f_{ij}(\inputVector_{ij}) \right) &=& \Var \left( \Esp \left[ f(\inputVector) \vert \inputVector_i, \inputVector_j \right] \right) - \Var \left( f_{i}(\inputVector_{i}) \right) - \Var \left( f_{j}(\inputVector_{j}) \right)  \\
  &=& \sobolClosedD{i,j} - \sobolClosedD{i} - \sobolClosedD{j}
\end{eqnarray*}

\end{remark}

\bigskip

For most problems, there are no closed formulas for the exact calculation of Sobol' indices, except for simple explicit models and for some classes of models, such as PCE-derived models \citep{Sudret2008}. In practice, the indices are estimated through sampling-based methods, where a finite sample of evaluations is employed for the estimation. In the sequel, we focus on pick-freeze (``PF'') schemes or estimators \citep{Sobol1993, Saltelli2008, Gamboa2013}, which have nice statistical properties and are applicable to general sets $I$.

\subsection{Pick-freeze estimators for scalar-valued functions}
\label{Subsec:pick_freeze_scalar}
Let $f : \R^{\nvar} \to \R$ be a scalar-valued function in $\mathbb{L}^2(\mu_{\inputVector})$. In the PF scheme, one considers two independent random vectors $X$ and $Z$ drawn from  $\probDist_X$. The model output is evaluated twice: firstly, by computing $Y= f(X)$; and secondly, by fixing (freezing) the coordinates of $X$ in $I$ and choosing (picking) the other coordinates (denoted ${(\cdot)}_{\exceptIndexSet}$) in $Z$, leading to $Y^{*} = f(\inputVector_{\indexSet}, Z_{\exceptIndexSet})$. By definition of $X$ and $Z$, we can see that $Y$ and $Y^*$ have the same probability distribution. But they are now dependent and their correlation is proved to be equal to the closed Sobol' index:

\begin{equation}
    \sobolClosed{\indexSet} = \frac{\Cov(Y, Y^{*})}{\Var ( Y )}
\label{Eq.closed_Sobol_variance_form} 
\end{equation}

\bigskip

This formula can be rewritten such that the first two moments of the model output involve the two copies $X$ and $Z$. Indeed, as $\mu_X = \mu_Z$, the first moment is written $f_0 = \Esp[Y] = \Esp[Y^{*}] = \Esp \left[ \frac{Y + Y^*}{2} \right]$. The latter expression is preferable for estimation, as the associated empirical estimator has a smaller quadratic risk. Similarly, the second moment is rewritten $\Esp[Y^2] = \Esp \left[ \frac{Y^2 + (Y^*)^2}{2} \right]$. Finally, using that $\Cov(Y, Y^*) = \Esp[Y Y^*] - f_0^2$ and $\Var(Y) = \Esp[Y^2] - f_0^2$, Eq. \ref{Eq.closed_Sobol_variance_form} is rewritten:

\begin{equation} 
    \sobolClosed{\indexSet} = \frac{ \Esp[ Y Y^{*} ] - \left( \Esp \left[ \frac{Y + Y^*}{2} \right] \right)^2 }{ \Esp \left[ \frac{Y^2 + (Y^*)^2}{2} \right] - \left( \Esp \left[ \frac{Y + Y^*}{2} \right] \right)^2}
\label{Eq.closed_Sobol_expected_value_form}  
\end{equation}

In practice, an empirical estimator must be computed. We consider independent samples $\sample{\inputVector}{1}, \dots, \sample{\inputVector}{\numberPickFreezeSamples}$ and $\sample{Z}{1}, \dots, \sample{Z}{\numberPickFreezeSamples}$
drawn from $\probDist_X$. 
The empirical versions of $Y$ and $Y^*$ are then given by the random variables
\begin{equation} \label{eq:PF_output}
Y_{(k)} = f(X^{(k)}), \qquad Y_{(k)}^* = f(X_I^{(k)}, Z_{\sim I}^{(k)}) \qquad (k=1, \dots, N).    
\end{equation}

\begin{defi}[Pick-freeze estimator of closed Sobol' indices for scalar-valued functions] 
\label{Def.pick_freeze_scalar}
The empirical estimator of the closed Sobol' index $\widehat{\sobolClosed{\indexSet}}^{pf}$  corresponding to Eq. \eqref{Eq.closed_Sobol_expected_value_form}, called Janon-Monod pick-freeze estimator \citep{Monod2006, Janon2014}, is given by 

\begin{equation}
    \label{Eq.pick_freeze_scalar_janon}
    \pfEstim{\sobolClosed{\indexSet}} = 
    \frac{\pfEstim{\sobolClosedD{I}}}{\pfEstim{\Var}}
\end{equation}
with 
\begin{eqnarray}
    \pfEstim{\sobolClosedD{I}} &=&
\frac{1}{\numberPickFreezeSamples} \sum_{k=1}^\numberPickFreezeSamples Y_{(k)} Y^{*}_{(k)} - \left( \pfEstim{f_0} \right)^2 \\
\pfEstim{\Var} &=& \frac{1}{\numberPickFreezeSamples} \sum_{k=1}^{\numberPickFreezeSamples} \left[ \frac{Y^2_{(k)} + (Y^{*}_{(k)})^2}{2} \right]- \left( \pfEstim{f_0} \right)^2
\end{eqnarray}
and 

\begin{equation} \label{eq:PF_f0}
\pfEstim{f_0} = \frac{1}{\numberPickFreezeSamples} \sum_{k=1}^{\numberPickFreezeSamples} \left[ \frac{Y_{(k)}+Y^{*}_{(k)}}{2} \right].
\end{equation}

\end{defi}

The PF estimator \eqref{Eq.pick_freeze_scalar_janon} has nice statistical properties: it is consistent, asymptotically normally distributed, and asymptotically efficient among a class of exchangeable variables, meaning that it has the smallest variance in this class when $N$ tends to infinity. We refer to \cite{Janon2014} for more details. 

Following Remark \ref{Remark.2nd_Sobol}, one can deduce a natural PF estimator for a Sobol' index associated to a set of variables $I$, from the (Janon-Monod) PF estimators of the closed Sobol' indices associated to subsets of $I$. For instance, for second-order indices, one can set
$$\pfEstim{\SobolUnnormalized_{i,j}} := \pfEstim{\sobolClosedD{i,j}} - \pfEstim{\sobolClosedD{i}} - \pfEstim{\sobolClosedD{j}}, \qquad \pfEstim{\SobolNormalized_{i,j}} = \frac{\pfEstim{\SobolUnnormalized_{i,j}}}{\pfEstim{\Var}}$$
Similarly, from Definition \ref{Def.Sobol_indices_scalar}, a natural candidate for a PF estimator of total Sobol' indices is written

\begin{equation}
\label{Eq.Sobol_total}
    \pfEstim{\sobolTot{\indexSet}} = 1 - \pfEstim{\sobolClosed{\exceptIndexSet}}
\end{equation}

A second option explores Jansen's formula \citep{Jansen1999}, which is recommended to evaluate total Sobol' indices \citep{Saltelli2010, DaVeiga2021}. The expression is given by Eq. \ref{Eq.Jansen}.

\begin{equation}
    \label{Eq.Jansen}
    \pfEstim{\sobolTot{\indexSet}} = \frac{\pfEstim{\sobolTotD{\indexSet}}}{ \pfEstim{\Var}}, \qquad \pfEstim{\sobolTotD{\indexSet}} = \frac{1}{2 \numberPickFreezeSamples} \sum_{k=1}^{\numberPickFreezeSamples} \left( Y_{(k)} - Y^{*}_{(k)} \right)^2 
\end{equation}

Indeed, the numerator is non-negative and has the same statistical properties as Janon-Monod estimator: consistency, asymptotic normality and asymptotic efficiency \citep{fruth2014total}. 

Actually, many PF estimators can be produced by combining different methods \citep{Saltelli2010, Owen2013, Azzini2021} for the numerator (partial variance) and denominator (overall variance). Here, we will focus on the Monod-Janon estimator \eqref{Eq.pick_freeze_scalar_janon} and Jansen estimator \eqref{Eq.Jansen} for the estimation of first-order and total Sobol' indices respectively. Nevertheless, we will see in the next section that our main result also applies to other PF estimators, as soon as they can be expressed with quadratic forms in the variables  $Y_{(1)},\dots, Y_{(N)}, Y_{(1)}^*,\dots, Y_{(N)}^*$ (see Remark~\ref{rem:pickFreezeFormulaAndQuadraticForm}).

\subsection{Pick-freeze estimators for vector-valued functions} \label{sec:PickFreezeVectorValued}

Here, we define natural extensions of the PF estimators to vector-valued functions. They will be useful in the sequel.

\begin{defi}[Unnormalized closed Sobol' index for vector-valued functions]
\label{Def.Sobol_indices_vector}
Let $\indexSet \in \mathcal{P}_{\nInputDimensions}$,  $f \in \mathbb{L}^2(\probDist)$, and $\inputVector$ be a random vector with law $\mu$. The unnormalized closed Sobol' index of $f(\inputVector) = Y$, denoted $\sobolClosedD{\indexSet}(Y)$, is defined as the covariance matrix of the random vector $\Esp[Y \vert \inputVector_{\indexSet}]$:

$$ \sobolClosedD{\indexSet} (Y) = \Cov(\Esp[ Y \vert \inputVector_{\indexSet}]) $$

\end{defi}

The PF estimators of the unnormalized closed or total Sobol' index and the overall variance are immediately extended for vector-valued function, by replacing products by dot products in Definition \ref{Def.pick_freeze_scalar}. However, some care is needed to define a ratio between such quantities, as they are now matrices. Actually, we will not pursue in this way, as we will not need it.

\begin{defi}[Pick-freeze estimator for vector-valued functions] 
\label{Def.Pick_freeze_vector}
The pick-freeze estimator of the unnormalized closed Sobol' index, total Sobol' index and the covariance matrix of a vector-valued function, are defined by the matrices
\begin{eqnarray*}
\pfEstim{\sobolClosedD{I}} &=& \frac{1}{\numberPickFreezeSamples} \sum_{k=1}^{\numberPickFreezeSamples} Y_{(k)} (Y^{*}_{(k)})^{\top} - \pfEstim{f_0}\left(\pfEstim{f_0}\right)^\top \\
\pfEstim{\sobolTotD{\indexSet}} &=& 
\frac{1}{2 \numberPickFreezeSamples} \sum_{k=1}^{\numberPickFreezeSamples} 
\left( Y_{(k)} - Y^{*}_{(k)} \right)
\left( Y_{(k)} - Y^{*}_{(k)} \right)^\top \\
\pfEstim{\Cov} &=& \frac{1}{\numberPickFreezeSamples} \sum_{k=1}^{\numberPickFreezeSamples} \left[ \frac{Y_{(k)}Y_{(k)}^\top + Y^*_{(k)}(Y^*_{(k)})^\top}{2} \right] - \pfEstim{f_0}\left(\pfEstim{f_0}\right)^\top
\end{eqnarray*}
with 
$$
\pfEstim{f_0} = \frac{1}{\numberPickFreezeSamples} \sum_{k=1}^{\numberPickFreezeSamples} \left[ \frac{Y_{(k)} + Y^{*}_{(k)}}{2} \right] 
$$

\label{Definition_Pick_freeze_vector}
\end{defi}

\section{Contributions}
\label{Sec.contributions}

The contributions of this work are focused on applications that employ basis expansions to reduce dimensionality. Let us consider that output data $\outputValues (\inputVector)$ are expanded in a functional basis of dimension $\nbasis$, where $(\cdot)_{\ell}$ represents the index of each output dimension, as follows:

\begin{equation}
    \label{Eq.basis_expansion}
    \outputValues (\inputVector) = \sum_{i=1}^{\nbasis} \coef_i(\inputVector) \basis_{i, \ell} 
\end{equation}

\noindent where $(c_1(X), \dots, c_m(X))$ are basis coefficients and $(\basis_{1, \ell}, \dots, \basis_{m, \ell})$ are basis components.
We set $\coef(X) = (c_1(X), \dots, c_m(X))^\top$ and $\basisVec = (\basis_{1, \ell}, \dots, \basis_{m, \ell})^\top$ the corresponding vectors.

The following proposition develops the expression of closed Sobol' indices considering the basis coefficients. 

\begin{prop}[Expression of closed Sobol' indices with basis coefficients] \label{prop:sobolWithBasis}
Let $\outputValues (\inputVector) = \basisVec^\top \ \coef(\inputVector)$ and, for all $\indexSet \subseteq \dimInt$,
\begin{equation} \label{eq:closedIndexWithCoef}
\sobolClosed{\indexSet}(\outputValues(\inputVector)) = \frac{\basisVec^\top \, \sobolClosedD{\indexSet}(\coef(\inputVector)) \, \basisVec}{\basisVec^\top \, \Cov(\coef(\inputVector)) \basisVec}
\end{equation}
\end{prop}

\begin{proof}

We should recall that if $\outputFunction$ is a square integrable random vector of size $\nbasis$, and $A$ is a (deterministic) $p \times \nbasis$ matrix, then $\Cov(A \outputFunction) = A \, \Cov(\outputFunction) \, A^\top$. For the numerator, by linearity of the conditional expectation, we have
$$ \Esp[ \outputValues(\inputVector) \vert \inputVector_{\indexSet}] = \basisVec^\top \, \Esp[ \coef(\inputVector) \vert \inputVector_{\indexSet}] $$

Applying the covariance formula with $A = \basisVec^\top$ and $\outputFunction = \Esp[ \coef(\inputVector) \vert \inputVector_{\indexSet}]$ gives the expression of the numerator:

$$ \sobolClosedD{\indexSet}(\outputValues(\inputVector)) = \basisVec^\top \, \sobolClosedD{\indexSet}(\coef(\inputVector)) \, \basisVec$$

Using the same logic for the denominator, we have immediately:
$$ \Var(\outputValues(\inputVector)) = \basisVec^\top \, \Cov(\coef(\inputVector)) \, \basisVec$$

This concludes the proof.
\end{proof}

Proposition \ref{prop:sobolWithBasis} shows that the closed Sobol' index of $\outputValues (\inputVector)$ can be computed directly from the random vector of its coefficients $\coef(\inputVector)$ in the basis $\basisVec$. Following Remark \ref{Remark.2nd_Sobol}, formula \eqref{eq:closedIndexWithCoef} can be extended to Sobol' indices of any order and total Sobol indices, as they can be expressed as linear combinations of closed Sobol indices. Notice further that \eqref{eq:closedIndexWithCoef} is similar to formula (11) of \citep{Li2020} written in the context of PCA for first-order Sobol' indices. Therein, a matrix $P$ is used, whose row $\ell$ contain the first PCA eigenvectors corresponding to $\basisVec^\top$ in our framework. With that notation, the numerator (resp. denominator) of \eqref{eq:closedIndexWithCoef} is equal to the coefficient $(\ell, \ell)$ of the diagonal matrix 
$\textrm{diag}(P \sobolClosedD{\indexSet} P^\top)$ 
(resp. 
$\textrm{diag}(P \Cov(\coef(\inputVector)) P^\top)$
), and we retrieve their expression.\\

We now focus on estimation of Sobol' indices. The next proposition shows that Proposition \ref{prop:sobolWithBasis} remains valid when we replace the theoretical quantities by their PF estimator, providing a formula to estimate the SMs.

\begin{prop}[Expression of pick-freeze estimators with basis coefficients] \label{prop:sobolHatWithBasis} For all $\indexSet \subseteq \dimInt$ and all $\ell \in \{1, \dots, L\}$, we have, with the notations of Definition~\ref{Def.pick_freeze_scalar} and Definition~\ref{Def.Pick_freeze_vector},
\begin{equation} \label{eq:closedIndexWithCoefEstim}
\pfEstim{\sobolClosed{\indexSet}}(\outputValues(\inputVector)) = \frac{\basisVec^\top \, \pfEstim{\sobolClosedD{\indexSet}}(\coef(\inputVector)) \, \basisVec}{\basisVec^\top \, \pfEstim{\Cov}(\coef(\inputVector)) \, \basisVec}    
\end{equation}
\end{prop}

\begin{proof}

Consider the same random vectors $\inputVector$ and $Z$ from Subsec. \ref{Subsec:pick_freeze_scalar} and let $Y = \outputValues(\inputVector) = \basisVec^\top \coef(\inputVector)$ and $Y^{*} = \basisVec^\top \coef(\inputVector_{\indexSet}, Z_{\exceptIndexSet})$. Let us denote $\mathcal{C} = \coef(\inputVector)$ and $\mathcal{C}^{*} = \coef(\inputVector_{\indexSet}, Z_{\exceptIndexSet})$, similarly to Definition \ref{Def.pick_freeze_scalar} for $f(\inputVector)$.\\
Let us first consider the PF estimator of the partial variance of $\outputValues(\inputVector)$.
By substituting it in Definition \ref{Def.pick_freeze_scalar}, we have 
$$\pfEstim{\sobolClosedD{\indexSet}}(\outputValues(\inputVector))
= Q(Y_{(1)}, \dots, Y_{(N)}, Y_{(1)}^*, \dots, Y_{(N)}^*)$$
where $Q$ is the quadratic form in the variables $Y_{(1)}, \dots, Y_{(N)}, Y_{(1)}^*, \dots, Y_{(N)}^*$
$$ Q(Y_{(1)}, \dots, Y_{(N)}, Y_{(1)}^*, \dots, Y_{(N)}^*) = \frac{1}{\numberPickFreezeSamples} \sum_{k=1}^{\numberPickFreezeSamples} Y_{(k)}Y_{(k)}^* - \left( \frac{1}{\numberPickFreezeSamples} \sum_{k=1}^{\numberPickFreezeSamples} \frac{Y_{(k)}+Y_{(k)}^*}{2}
\right)^2$$
Notice that the definition of $Q$ can be extended to vectors by replacing products with dot products:
\begin{eqnarray*}
    Q(Y_{(1)}, \dots, Y_{(N)}, Y_{(1)}^*, \dots, Y_{(N)}^*) &=& \frac{1}{\numberPickFreezeSamples} \sum_{k=1}^{\numberPickFreezeSamples} Y_{(k)}(Y_{(k)}^*)^\top \\
    &-& \left( \frac{1}{\numberPickFreezeSamples} \sum_{k=1}^{\numberPickFreezeSamples} \frac{Y_{(k)}+Y_{(k)}^*}{2}
\right)\left( \frac{1}{\numberPickFreezeSamples} \sum_{k=1}^{\numberPickFreezeSamples} \frac{Y_{(k)}+Y_{(k)}^*}{2}
\right)^\top
\end{eqnarray*}
With this extended definition, we have, from Definition~\ref{Def.Pick_freeze_vector}, 
$$ \pfEstim{\sobolClosedD{\indexSet}}(\coef(\inputVector)) = 
Q(\mathcal{C}_{(1)}, \dots, \mathcal{C}_{(N)}, \mathcal{C}_{(1)}^*, \dots, \mathcal{C}_{(N)}^*)$$
Furthermore, as $Q$ is a quadratic form, a direct computation shows that
\begin{eqnarray*}
  Q(Y_{(1)}, \dots, Y_{(N)}, Y_{(1)}^*, \dots, Y_{(N)}^*) &=&
  Q(\basisVec^\top \mathcal{C}_{(1)}, \dots, \basisVec^\top \mathcal{C}_{(N)}, \basisVec^\top \mathcal{C}_{(1)}^*, \dots, \basisVec^\top \mathcal{C}_{(N)}^*) \\
  &=& \basisVec^\top Q(\mathcal{C}_{(1)}, \dots, \mathcal{C}_{(N)}, \mathcal{C}_{(1)}^*, \dots, \mathcal{C}_{(N)}^*) \basisVec
\end{eqnarray*}
 
Finally, we obtain:
$$\pfEstim{\sobolClosedD{I}}(\outputValues(\inputVector)) = \basisVec^\top \pfEstim{\sobolClosedD{I}}(\coef(\inputVector)) \basisVec $$
The same method can be used for the PF estimator of the overall variance of $\outputValues(\inputVector)$, remarking that it is also defined as a quadratic form in the variables $Y_{(1)}, \dots, Y_{(N)}, Y_{(1)}^*, \dots, Y_{(N)}^*$. This leads to the announced formula.
\end{proof}

\begin{remark}[Extension to other pick-freeze estimators]
\label{rem:pickFreezeFormulaAndQuadraticForm}
Looking at the proof of Proposition~\ref{prop:sobolHatWithBasis}, we see that the result can be extended to all PF estimators that can be expressed with quadratic forms in the variables of the two samples. For instance, for the estimator \ref{Eq.Jansen} of total Sobol' indices, we have, with the notations of Definition~\ref{Def.pick_freeze_scalar} and Definition~\ref{Def.Pick_freeze_vector}:
\begin{equation}
\label{eq:totalIndexWithCoefEstim}
 \pfEstim{\sobolTot{\indexSet}}(\outputValues(\inputVector)) = \frac{\basisVec^\top \, \pfEstim{\sobolTotD{\indexSet}}(\coef(\inputVector)) \, \basisVec}{\basisVec^\top \, \pfEstim{\Cov}(\coef(\inputVector)) \, \basisVec}   
\end{equation}

\end{remark}

Proposition~\ref{prop:sobolHatWithBasis} allow to estimate the SMs of closed Sobol' indices in two approaches and obtain the exact same result, either using the Definition \ref{Def.pick_freeze_scalar} for each output dimension (\textit{dimension-wise} approach) or using formula \eqref{eq:closedIndexWithCoefEstim} (\textit{basis-derived} approach). 
This also applies for the SMs of Sobol' indices and total Sobol' indices, following Remark~\ref{Remark.2nd_Sobol} and Remark~\ref{rem:pickFreezeFormulaAndQuadraticForm}.
In the \textit{dimension-wise} approach, only scalar-valued functions are considered, but the computation of Sobol' indices must be done for each pixel $\ell$. In the \textit{basis-derived} approach, a matrix containing the Sobol' index of the vector of coefficients is first computed. As this matrix does not depend on $\ell$, it can be stored and reused to compute the SM by simple vector-matrix multiplications. Furthermore, as this matrix is generally of small size, equal to the number of basis functions, we can expect a substantial gain in the computational time.

To confirm this intuition, we now quantify the computational cost of the two approaches. We will assume the same cost for additions and multiplications. 

\bigskip

\begin{prop}[Computational cost to compute $\sobolClosed{I}(\outputValues(\inputVector))$] 
\label{prop:ratio_costs}
Assume that the PF samples $\sample{\inputVector}{1}, \dots, \sample{\inputVector}{\numberPickFreezeSamples}$, $\sample{Z}{1}, \dots, \sample{Z}{\numberPickFreezeSamples}$
and the corresponding output values
$c(X^{(k)})$ and $c(X_I^{(k)}, Z_{\sim I}^{(k)})$ have been computed for $k=1, \dots, N$, with $N \gg 1$. Let $H(x_1,x_2) = \frac{2x_1 x_2}{x_1 + x_2}$ be the harmonic mean of two positive real numbers $x_1, x_2$. Then the ratio of computational costs between the \textit{dimension-wise} (DW) and \textit{basis-derived} (BD) approaches is equal to
\begin{equation}
    \frac{(2 \nbasis + 4) \numberPickFreezeSamples \nOutputDimensions}{ 4 \nbasis^2 \numberPickFreezeSamples + \nbasis \numberPickFreezeSamples + \nbasis^2 + \nbasis \nOutputDimensions (\nbasis + 1)}
    > \frac{H(2\numberPickFreezeSamples, \nOutputDimensions)}{3 \nbasis}
    \label{Eq.ratio_costs}
\end{equation}

\begin{proof}
The \textit{dimension-wise} approach first requires obtaining $\nOutputDimensions$ scalar output values by decoding the basis coefficients (Eq. \ref{Eq.basis_expansion}) and then performing scalar-valued pick-freeze estimation over the $\nOutputDimensions$ output dimensions through Definition~\ref{Def.pick_freeze_scalar}, using $Y = \outputValues(\inputVector) = \basisVec^\top \coef(\inputVector)$ and $Y^{*} = \basisVec^\top \coef(\inputVector_{\indexSet}, Z_{\exceptIndexSet})$. Next, the \textit{basis-derived} approach performs the vector-valued pick-freeze estimation (Definition~\ref{Def.Pick_freeze_vector}) considering the $\nbasis$-sized vectors $Y=\coef(\inputVector)$ and $Y^{*} = \coef(\inputVector_{\indexSet}, Z_{\exceptIndexSet})$. Then Proposition~\ref{prop:sobolHatWithBasis} is applied. Tables \ref{Tab.:computational_costs_dimension} and \ref{Tab.:computational_costs_basis} summarize the computational costs related to each mathematical expression calculated in each approach. For the last two operations of Table \ref{Tab.:computational_costs_basis}, we expanded the expression of form $x^\top A x$ as follows

$$ x^\top A x = \sum_{i=1}^m x_i^2 A_{i,i} + 2 \sum_{1 \leq i < j \leq m} x_i x_j A_{i,j}. $$
\\
where $A$ is a symmetric matrix of size $(\nbasis, \nbasis)$ and $x$ is a vector of length $\nbasis$. The computation of $x^\top A x$ can be done with approximately $\frac{3}{2} m(m+1)$ operations.

The ratio of computational costs is immediately deduced, using that $3 \nbasis^2$ can be neglected relatively to $2\nbasis(3\nbasis+1) \numberPickFreezeSamples$ as $N \gg 1$. The lower bound in \eqref{Eq.ratio_costs} is due to the following inequalities (which are sharp when $m$ is large)
$$ \frac{2\nbasis (3 \nbasis + 1)}{\nbasis + 2} < 6 \nbasis, \qquad 
\frac{3\nbasis (\nbasis + 1)}{\nbasis + 2} < 3 \nbasis,
$$
which implies
 $$ \frac{\costDW}{\costBD} \approx \frac{4(\nbasis + 2)
 \numberPickFreezeSamples \nOutputDimensions}{ 2 \nbasis (3 \nbasis + 1) \numberPickFreezeSamples + 3 \nbasis (\nbasis + 1) \nOutputDimensions }
> \frac{
 4 \numberPickFreezeSamples \nOutputDimensions}{ 6 \nbasis \numberPickFreezeSamples + 3 \nbasis \nOutputDimensions }
 = \frac{H(2N, L)}{ 3\nbasis}
$$

\begin{table}[]
    \renewcommand{\arraystretch}{2.0}
    \begin{tabular}{>{\arraybackslash}m{7.5cm}>{\centering\arraybackslash}m{4.5cm}}
    \toprule
        \textbf{\textit{Dimension-wise} approach} & \textbf{Cost} \\
        \toprule
        \multirow{5}{*}{}
        Set $Y = \outputValues(\inputVector)
        = \sum_{i=1}^{\nbasis} \coef_i(\inputVector) \basis_{i, \ell}$ & \\
        Compute $Y_{(k)}, Y_{(k)}^*, \quad k=1, \dots, N$ & $4 \nbasis \numberPickFreezeSamples$ \\
        Compute $\alpha := \frac{1}{\numberPickFreezeSamples} \sum_{k=1}^\numberPickFreezeSamples Y_{(k)} Y^{*}_{(k)}$ & $2\numberPickFreezeSamples$ \\
        Compute $ \beta := \frac{1}{2\numberPickFreezeSamples} \sum_{k=1}^{\numberPickFreezeSamples} \left( Y^2_{(k)} + (Y^{*}_{(k)})^2 \right)$ \vspace{5pt} & $4\numberPickFreezeSamples$ \vspace{5pt}\\
        Compute $\gamma := \frac{1}{2\numberPickFreezeSamples} \sum_{k=1}^{\numberPickFreezeSamples} \left( Y_{(k)}+Y^{*}_{(k)} \right)$ & $2\numberPickFreezeSamples$ \\
        Deduce $\pfEstim{\sobolClosed{I}} (\outputValues(\inputVector)) = \frac{\alpha - \gamma^2}{\beta - \gamma^2}$ & - \\
        \cline{2-2}
        \textbf{Total Cost ($\ell$ fixed)} &  $4( \nbasis + 2) \numberPickFreezeSamples$ \\
        \textbf{Total Cost (for $\ell = 1, \dots, \nOutputDimensions$)} &  $
        \costDW = 4( \nbasis + 2) \numberPickFreezeSamples \nOutputDimensions$ \\
        \bottomrule
    \end{tabular}
    \caption{\textit{Dimension-wise} approach expressions and respective computational costs in function of number of basis components $\nbasis$, length of PF samples $\numberPickFreezeSamples$ and number of output dimensions $\nOutputDimensions$.} 
    \label{Tab.:computational_costs_dimension}
\end{table}

\begin{table}[]
\small
    \centering
    \renewcommand{\arraystretch}{2.0}
    \begin{tabular}{>{\arraybackslash}m{8.5cm}>{\centering\arraybackslash}m{3.5cm}}
    \toprule
        \textbf{\textit{Basis-derived} approach } & \textbf{Cost} \\
        \toprule
        \multirow{8}{*}{} 
        Set $Y = \coef(\inputVector) = (c_1(X), \dots, c_m(X))^\top$ & \\
        Compute $\displaystyle A := \frac{1}{\numberPickFreezeSamples} \sum_{k=1}^{\numberPickFreezeSamples} Y_{(k)} (Y^{*}_{(k)})^{\top}$ & $2 \nbasis^2 \numberPickFreezeSamples$ \\
        Compute $\displaystyle B:=\frac{1}{2\numberPickFreezeSamples} \sum_{k=1}^{\numberPickFreezeSamples} \left[Y_{(k)}Y_{(k)}^\top + Y^*_{(k)}(Y^*_{(k)})^\top \right]$ & 
        $4 \nbasis^2 \numberPickFreezeSamples$
        \\
        Compute $\gamma := \frac{1}{2\numberPickFreezeSamples} \sum_{k=1}^{\numberPickFreezeSamples} \left[ Y_{(k)} + Y^{*}_{(k)} \right]$ & $2 \nbasis \numberPickFreezeSamples$ \\
        Deduce $C = \gamma \gamma^\top$,  $\pfEstim{\sobolClosedD{\indexSet}}(\coef(\inputVector))  = A - C$ and $\pfEstim{\Cov}(\coef(\inputVector)) = B - C$
        & $3 \nbasis^2$\\
        Store $\pfEstim{\sobolClosedD{\indexSet}}(\coef(\inputVector))$ and $\pfEstim{\Cov}(\coef(\inputVector))$ & \\
        For $\ell = 1, \dots, \nOutputDimensions$, compute $\basisVec^\top \, \pfEstim{\sobolClosedD{\indexSet}}(\coef(\inputVector)) \, \basisVec$ & 
        $\displaystyle \approx \frac{3}{2}m(m+1) L $ \\
        For $\ell = 1, \dots, \nOutputDimensions$, compute $\basisVec^\top \, \pfEstim{\Cov}(\coef(\inputVector)) \, \basisVec$  \vspace{5pt} & 
        $\displaystyle \approx \frac{3}{2}m(m+1) L $ 
        \vspace{15pt} \\
        \cline{2-2}
        \vspace{10pt}
        \textbf{Total Cost}:
        $ \costBD \approx 2 \nbasis (3 \nbasis + 1) \numberPickFreezeSamples 
         + 3 \nbasis (\nbasis + 1) \nOutputDimensions$ \\
        \bottomrule
    \end{tabular}
    
    \caption{\textit{Basis-derived} approach expressions and respective computational costs in function of number of basis components $\nbasis$, length of PF samples $\numberPickFreezeSamples$ and number of output dimensions $\nOutputDimensions$. 
    }
    \label{Tab.:computational_costs_basis}
\end{table}

\end{proof}

\end{prop}

Proposition~\ref{prop:ratio_costs} shows that the ratio $\frac{\costDW}{\costBD}$ increases with $\nOutputDimensions$ and decreases with $\nbasis$. This is expected since $\costDW$ requires iterating over the output dimensions $\nOutputDimensions$ whereas $\costBD$ depends on the vector of basis coefficients of size $\nbasis$.
In practice, if the dimension reduction is effective, we may expect that $\nbasis \ll \nOutputDimensions$. We may also choose $\numberPickFreezeSamples$ such that $\nbasis \ll \numberPickFreezeSamples$. Then, $
H(2\numberPickFreezeSamples, \nOutputDimensions) > \min(2\numberPickFreezeSamples, \nOutputDimensions)$, implying that
$$  \frac{\costDW}{\costBD} 
    > \frac{\min(2\numberPickFreezeSamples, \nOutputDimensions)}{3 \nbasis} \gg 1
$$
This confirms the intuition of substantial computational gain of the \textit{basis-derived} approach over the \textit{dimension-wise} approach while obtaining the same exact result.

\section{Applications}
\label{Sec.applications}

In this section, we present applications of GSA in cases with high dimensional outputs (time series and maps) using the contribution described in this work, i.e. PF scheme applied directly in basis components. The goal is to obtain the \textit{Sensitivity Maps} (SMs). The treatment of functional outputs follows a general methodology, explained extensively in \citet{Marrel2011, Nagel2020, Li2020, Perrin2021}. Basically, two steps are followed. First, output data are expanded in a functional basis (Eq. \ref{Eq.basis_expansion}), which allows us to represent data with few $\nbasis$ dimensions instead of $\nOutputDimensions$ dimensions. Second, a metamodel is trained for each basis coefficient, enabling the fast and reasonably accurate prediction of coefficients. Therefore, the PF scheme becomes feasible with low computational resources. If one wants to predict the functional data represented in high-dimension $\nOutputDimensions$, the coefficients can be predicted by the metamodels and then original high-dimensional data can be recovered through a scalar product operation with the basis components.

In the following applications, we employed the Principal Component Analysis (PCA) and Gaussian Process Regression (GPR) methods as the basis-expansion and metamodeling techniques, respectively. The implementation was done in Python language, using the ``scikit-learn'' \citep{Scikitlearn2011} and ``gpflow'' \citep{GPflow2017} libraries for the PCA and the GPR procedures, respectively. The scripts are publicly available in \url{https://osf.io/njgzt/?view_only=9d7d6f02c1e447059bf2792aa34e1db0}. To verify the accuracy of the metamodels' predictions, we employed the $Q^2$ metric \citep{Marrel2011} as indicated by Eq. \ref{eq:q_squared}.

\begin{equation}
    Q^2 = 1 - \frac{\Esp_{\ell} \{ \Esp_{\inputVector} \left[ (\outputValues(\inputVector) - \widehat{\outputValues}(\inputVector))^2 \right]  \} }{\Esp_{\ell} \{ \Var_{\inputVector} [ \outputValues(\inputVector) ] \}}
    \label{eq:q_squared}
\end{equation}

where $\widehat{\outputValues}$ are metamodels' predictions. Here $\Esp_\ell$ denotes the expectation for the uniform probability distribution on $\{1, \dots, L\}$, equal to the average on $\ell$. In practice $\Esp_X$ and $\Var_X$ are approximated by computing the average on a sample drawn from the probability distribution of $X$.
If $Q^2 > 0$, the metamodels' prediction is a better prediction than the average value of the observation. The closest to $1$, the better the prediction accuracy.

\bigskip

Finally, since this work deals with high-dimensional output models, sensitivity indices are distributed along the domain (space or time series). To assess the global influence of each input variable in a single concise index, we employed the Generalized Sensitivity Index (GSI) \citep{Lamboni2011, Perrin2021}, indicated by Eq. \ref{eq:gsi}.

\begin{equation}
    \text{GSI}^c_I = \frac{\sum^{\nbasis}_{k=1} \lambda_k \sobolClosed{I,k}}{\sum^{\nbasis}_{k=1} \lambda_k}
    \label{eq:gsi}
\end{equation}

where $\lambda_k$ is the eigenvalue of the k-th basis component and $\sobolClosed{I,k}$ is the closed Sobol' index of the k-th basis coefficient. The GSI index is defined in a similar way for total Sobol' indices, replacing $\sobolClosed{I,k}$ by $\sobolTot{I,k}$ in \eqref{eq:gsi}.

\bigskip

The following Subsection \ref{Subsec.campbell2d} details the application of GSA in an analytical function (Campbell2D function). Then, Subsection \ref{Subsec.dam_break} describes the case of an idealized gradual dam-break flow case of non-Newtonian fluids, where GSA is applied to the observed time series - the maximum position of the wave (runout) over time.

\subsection{Analytical case - Campbell2D function}
\label{Subsec.campbell2d}

The analytical test case presented here is the Campbell2D function \citep{Marrel2011}. The function has 8 inputs ($ \nInputDimensions = 8$), which can range from $-1$ to $5$, and produces a spatial map $\outputValues(\inputVector)$ as output. Equation \ref{Eq.campbell1} shows the mapping of inputs to spatial outputs. Figures \ref{Fig.campbell_example_1}, \ref{Fig.campbell_example_2}, \ref{Fig.campbell_example_3} shows Campbell2D outputs for different input values, where the spatial domain is discretized on an uniform grid of size $64 \times 64$.

\begin{equation}
\begin{aligned}
f\colon \quad \quad \left[ -1, 5 \right]^2 & \longrightarrow\mathbf L^2 \left( \left[ -90, 90\right]^2 \right) \\
\inputVector = \left( x_1, ..., x_8 \right) & \longmapsto \outputValues \left( \inputVector \right)
\end{aligned}
\end{equation}
\label{Eq.campbell1}

where $\spatialLocation = (z_1, z_2) \in \left[-90,90 \right]^2$.

\bigskip

\begin{figure}[H]
    \centering
    \begin{subfigure}[b]{0.32\textwidth}
        \centering
        \includegraphics[width=\textwidth]{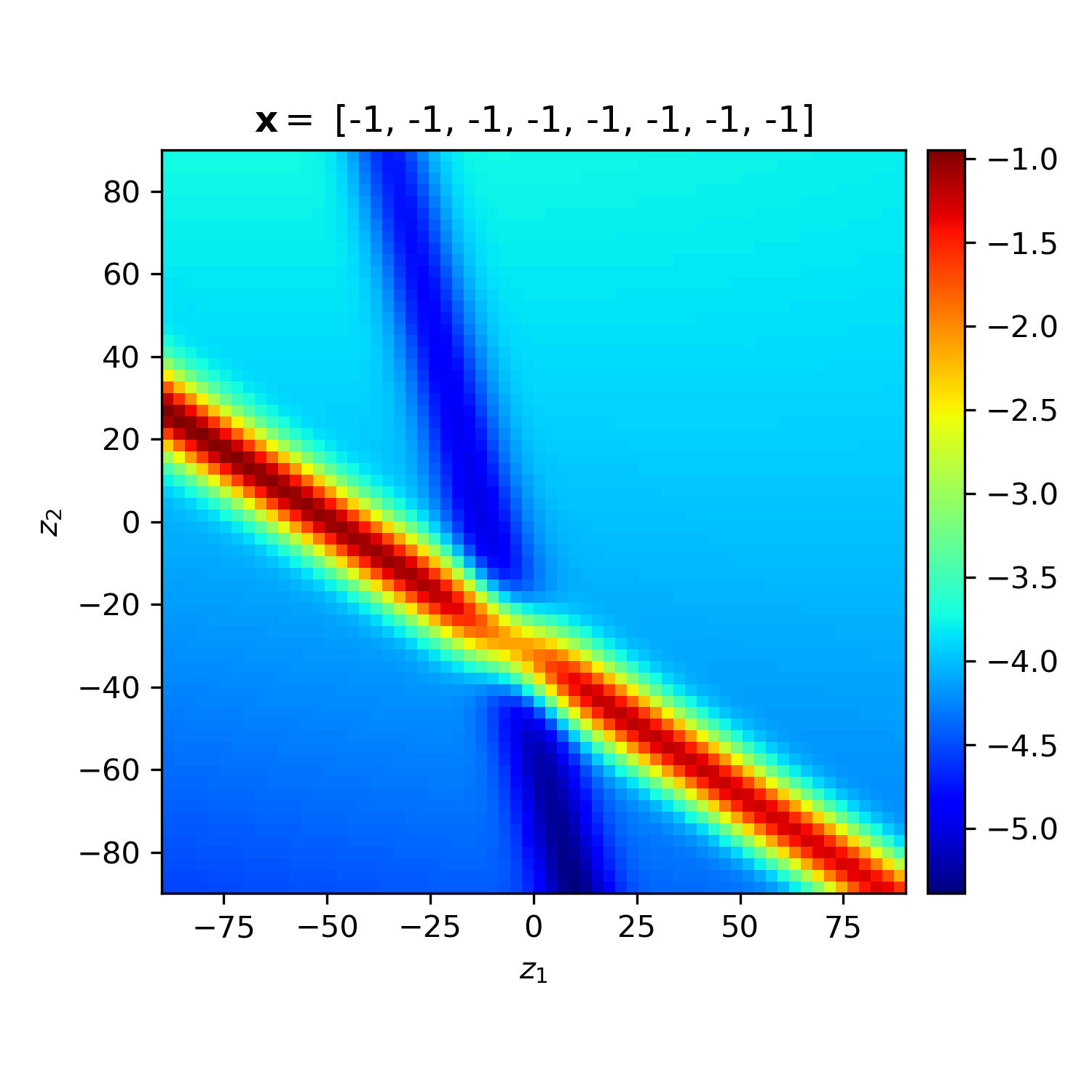}
        \caption{$\inputVector=$ [-1,-1,-1,-1,-1,-1,-1,-1].}
        \label{Fig.campbell_example_1}
    \end{subfigure}
    \hfill
    \begin{subfigure}[b]{0.32\textwidth}
        \centering
        \includegraphics[width=\textwidth]{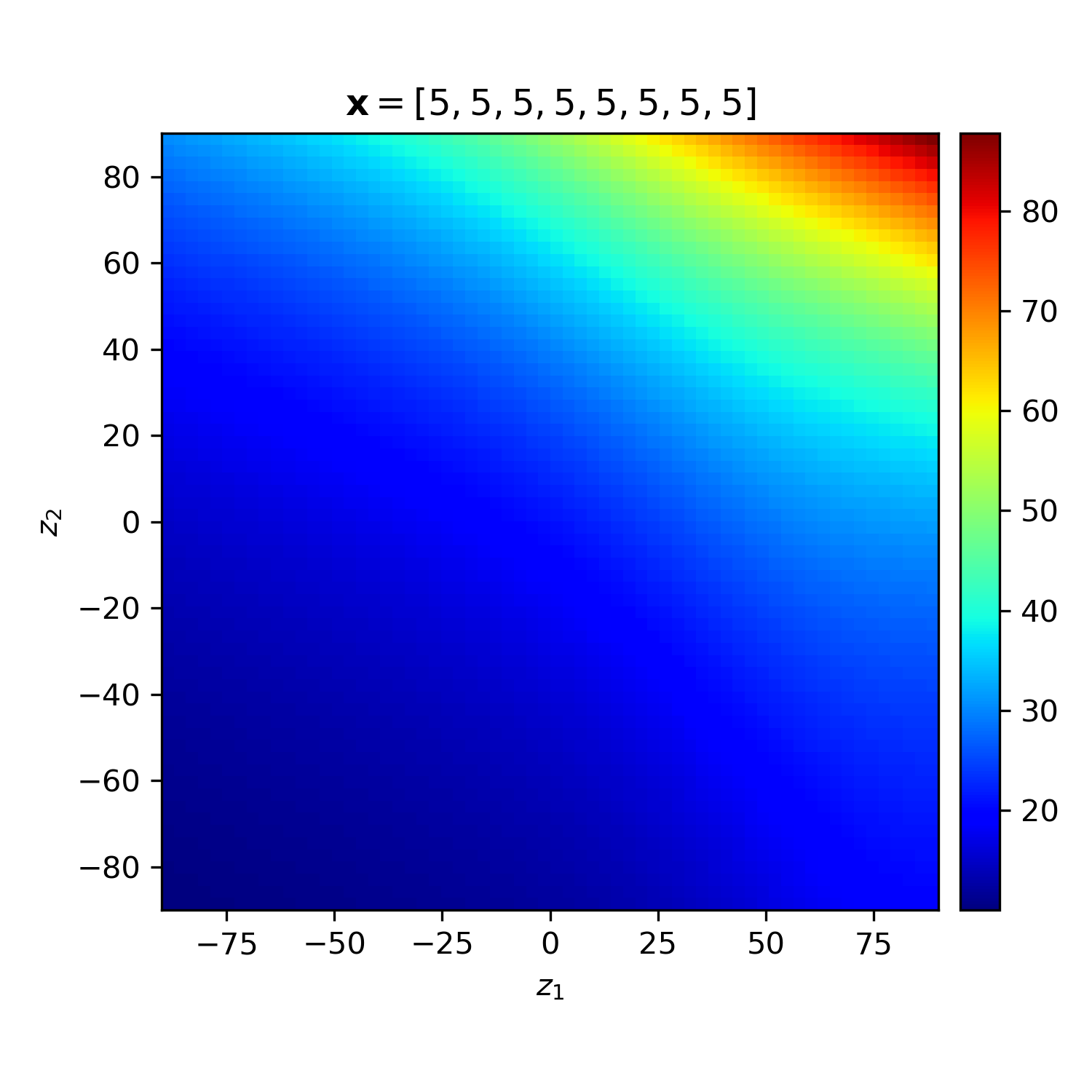}
        \caption{$\inputVector=$ [5,5,5,5,5,5,5,5].}
        \label{Fig.campbell_example_2}
    \end{subfigure}
    \hfill
    \begin{subfigure}[b]{0.32\textwidth}
        \centering
        \includegraphics[width=\textwidth]{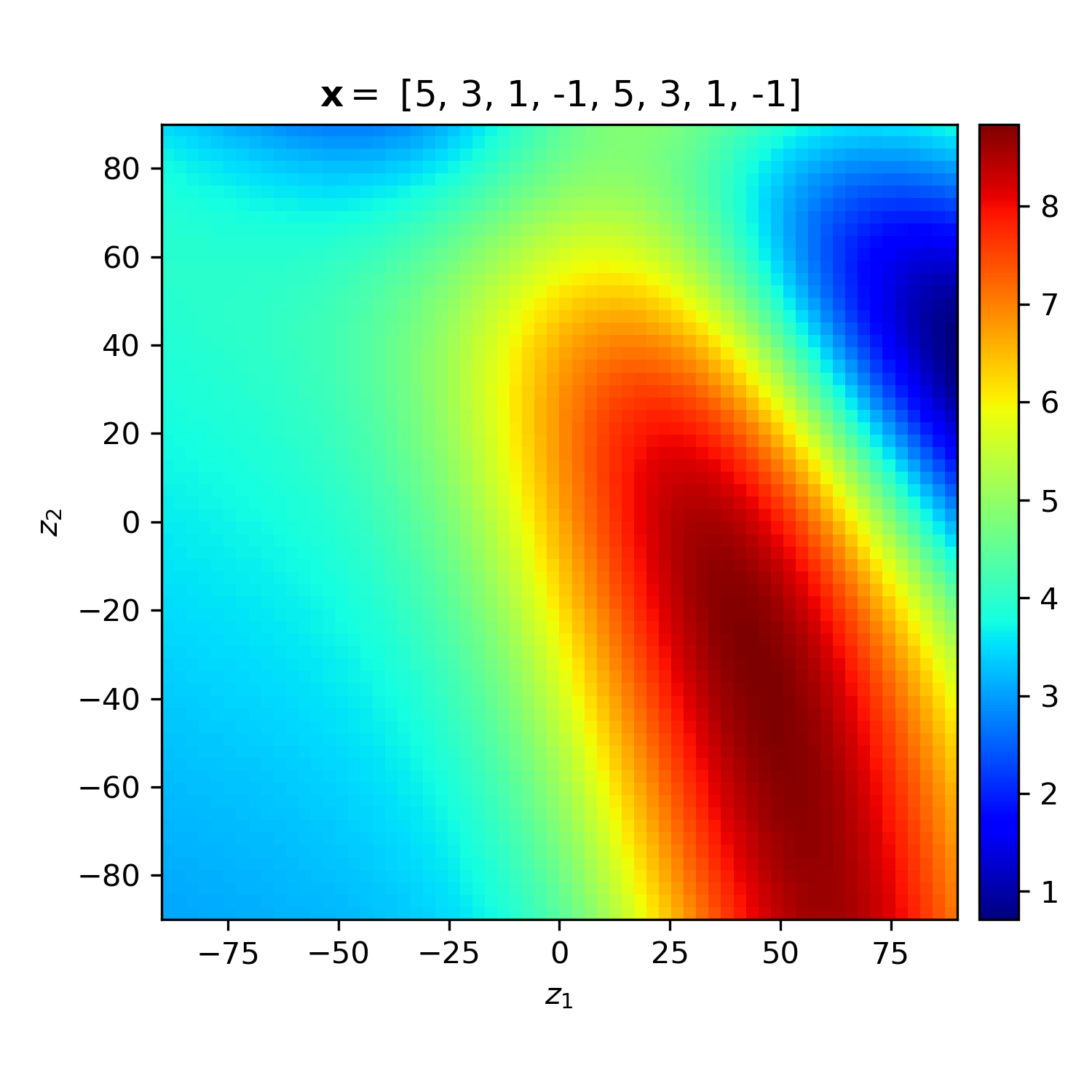}
        \caption{$\inputVector=$ [5,3,1,-1,5,3,1,-1].}
        \label{Fig.campbell_example_3}
    \end{subfigure}
    \caption{Examples of Campbell2D spatial outputs.}
    \label{Fig.campbell_examples}
\end{figure}

Using the Latin Hypercube Sampling (LHS) technique, a Design of Experiment (DoE) of size $200$ was generated considering uniform probability distributions for input variables. For the basis expansion, PCA was applied using $\nbasis = 7$ basis components, which account for $99.24 \%$ of the total variance. Then, each basis coefficient was metamodeled by a GPR using the Matérn 5/2 kernel and the ``L-BFGS-B'' optimizer for Maximum Likelihood Estimation. To evaluate the accuracy of the metamodels, the $Q^2$ metric (Eq. \ref{eq:q_squared}) was employed considering a validation set of size $50$. Figure \ref{Fig.q2_campbell} shows $Q^2 > 0.95$, which implies good predictions from the metamodels. Therefore, we can ensure reasonably accurate predictions for the basis coefficients using the trained metamodels.

\begin{figure}[H]
    \centering
        \includegraphics[width=0.85\textwidth]{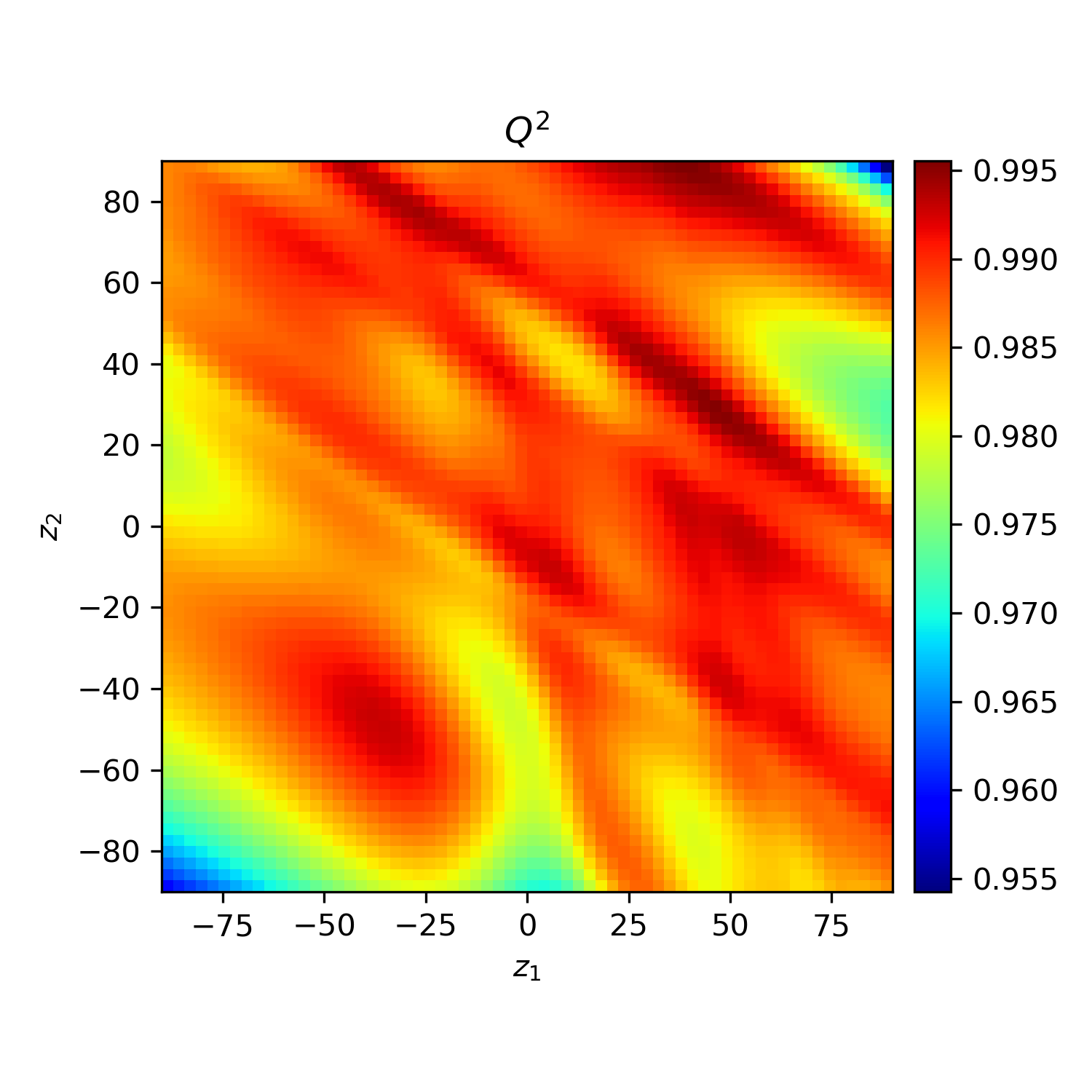}
    \caption{$Q^2$ metrics for evaluation of metamodels' accuracy.}
    \label{Fig.q2_campbell}
\end{figure}

To estimate the SMs, two methods were employed: the direct method and the \textit{basis-derived} method. For the direct method, the outputs $\outputValues(\inputVector)$ of Campbell2D function were calculated using the theoretical equation and applied directly in Definition \ref{Def.pick_freeze_scalar} for each output dimension. For the \textit{basis-derived} solution, the predictions of basis coefficients from GPR metamodels were generated and applied in Definition \ref{Def.Pick_freeze_vector}. Both of them were computed using PF samples of size $5\,000$. The bootstrap method was employed to evaluate the mean result and the correspondent standard deviation, using $50$ bootstrap replicates. Figures \ref{Fig.first_order_var2} and \ref{Fig.first_order_var6} show the first-order SMs for input variables $x_2$ and $x_6$, both referring to the direct and \textit{basis-derived} methods. 
In the same direction, Figs. \ref{Fig.total_order_var4} and \ref{Fig.total_order_var8} show the total order SMs for input variables $x_4$ and $x_8$, respectively. The results of the \textit{basis-derived} method present a good agreement with those of the direct method, with errors mostly ranging from $5$ to $15\%$. We can notice some zones with larger errors ($50 \%$ for Fig. \ref{Fig.first_order_var2} and $30 \%$ for Fig. \ref{Fig.first_order_var6}), associated with low-value Sobol' indices. Since we are considering the predicted SM with bootstrap mean, the error can be sensitive to the standard deviation around the mean due to the low-value indices.

\begin{figure}[H]
    \centering
    \begin{subfigure}[b]{0.32\textwidth}
        \centering
        \includegraphics[width=\textwidth]{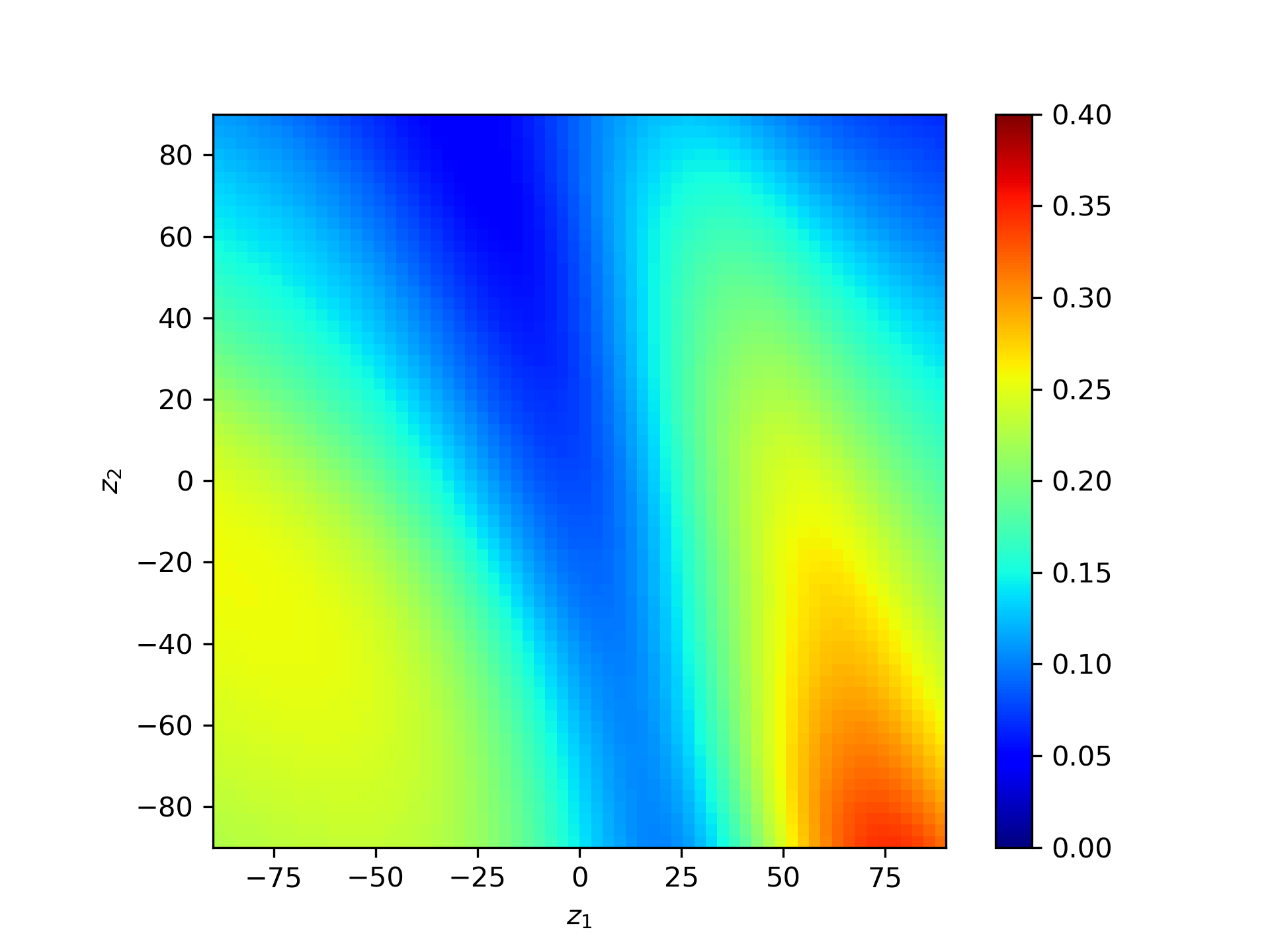}
        \caption{Direct method.}
    \end{subfigure}
    \hfill
    \begin{subfigure}[b]{0.32\textwidth}
        \centering
        \includegraphics[width=\textwidth]{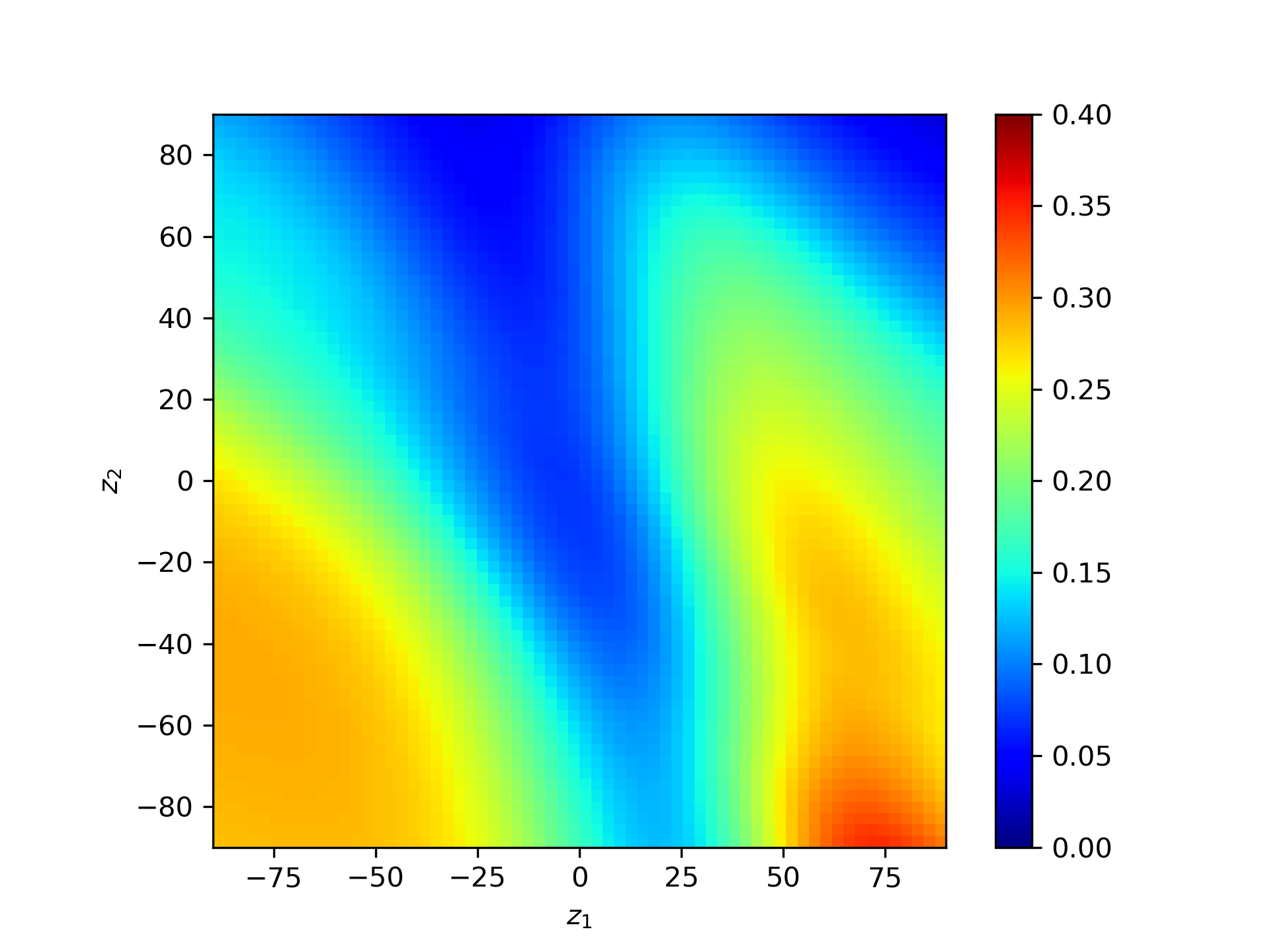}
        \caption{\textit{Basis-derived} method.}
    \end{subfigure}
    \hfill
    \begin{subfigure}[b]{0.32\textwidth}
        \centering
        \includegraphics[width=\textwidth]{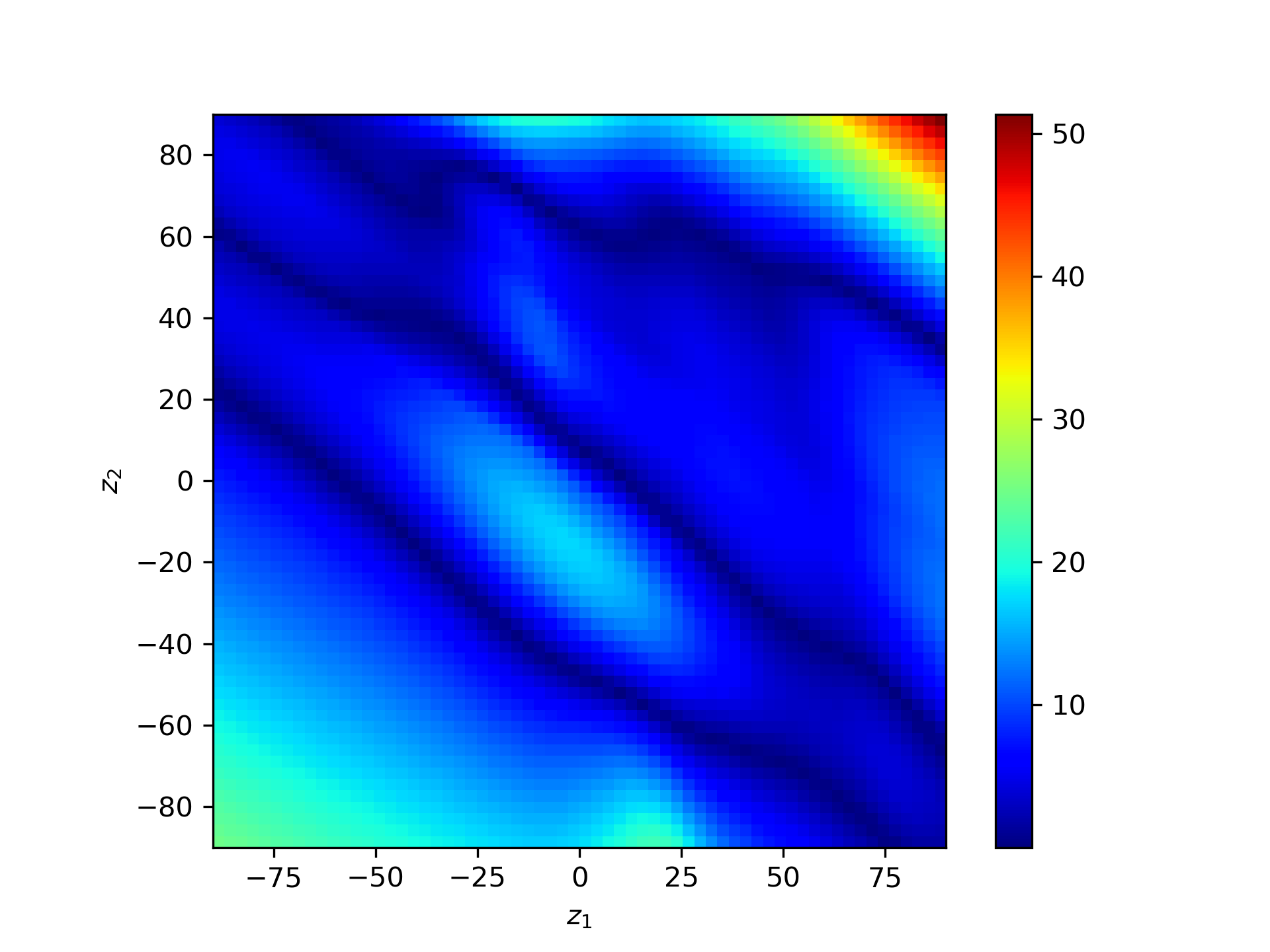}
        \caption{Relative errors.}
    \end{subfigure}
    \caption{Comparison of the first-order SMs for $x_2$ obtained with the direct and \textit{basis-derived} methods. In (c), the color scale values correspond to percentages of the direct method result represented in (a).}
    \label{Fig.first_order_var2}
\end{figure}

\begin{figure}[H]
    \centering
    \begin{subfigure}[b]{0.32\textwidth}
        \centering
        \includegraphics[width=\textwidth]{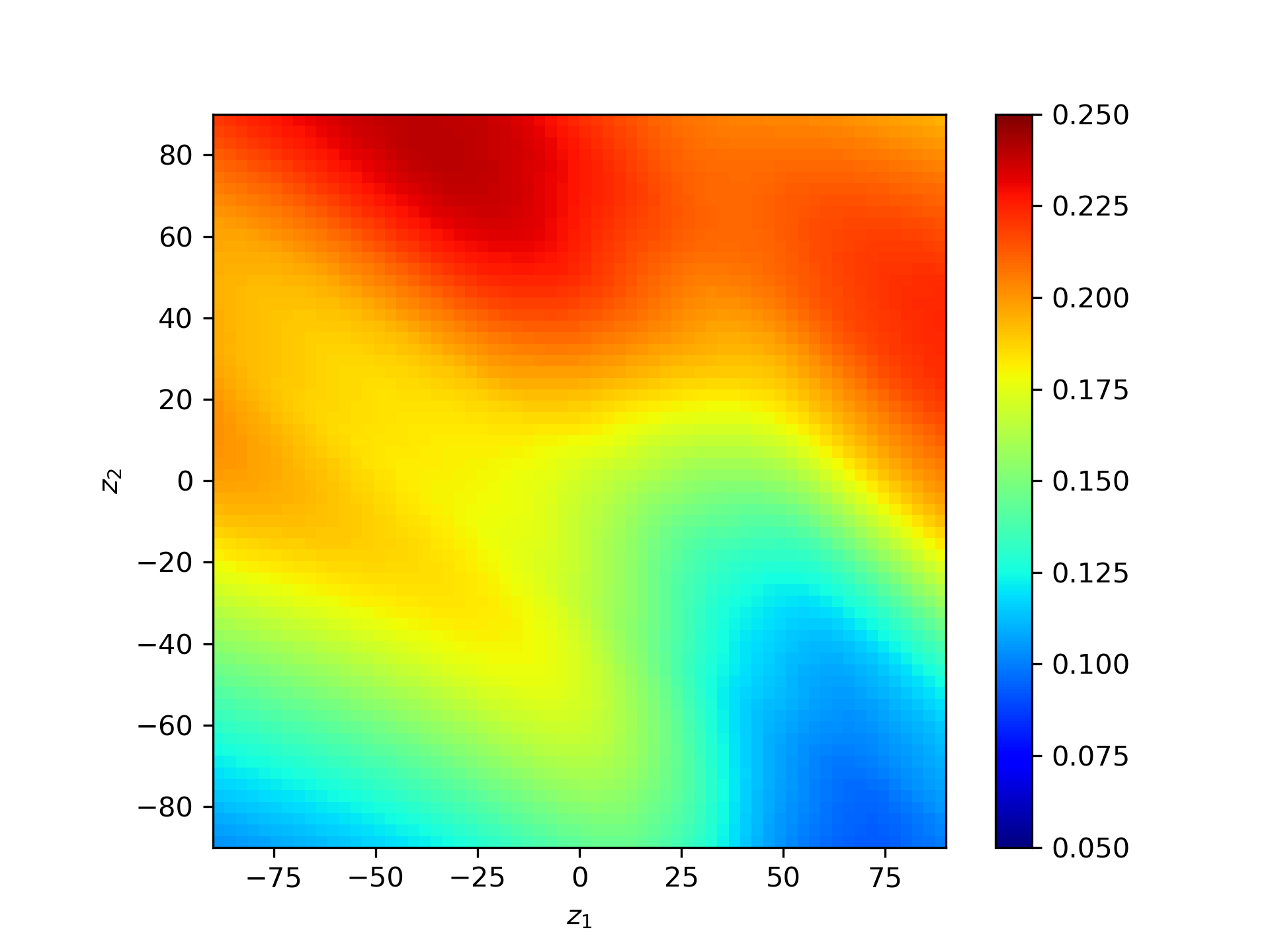}
        \caption{Direct method.}
    \end{subfigure}
    \hfill
    \begin{subfigure}[b]{0.32\textwidth}
        \centering
        \includegraphics[width=\textwidth]{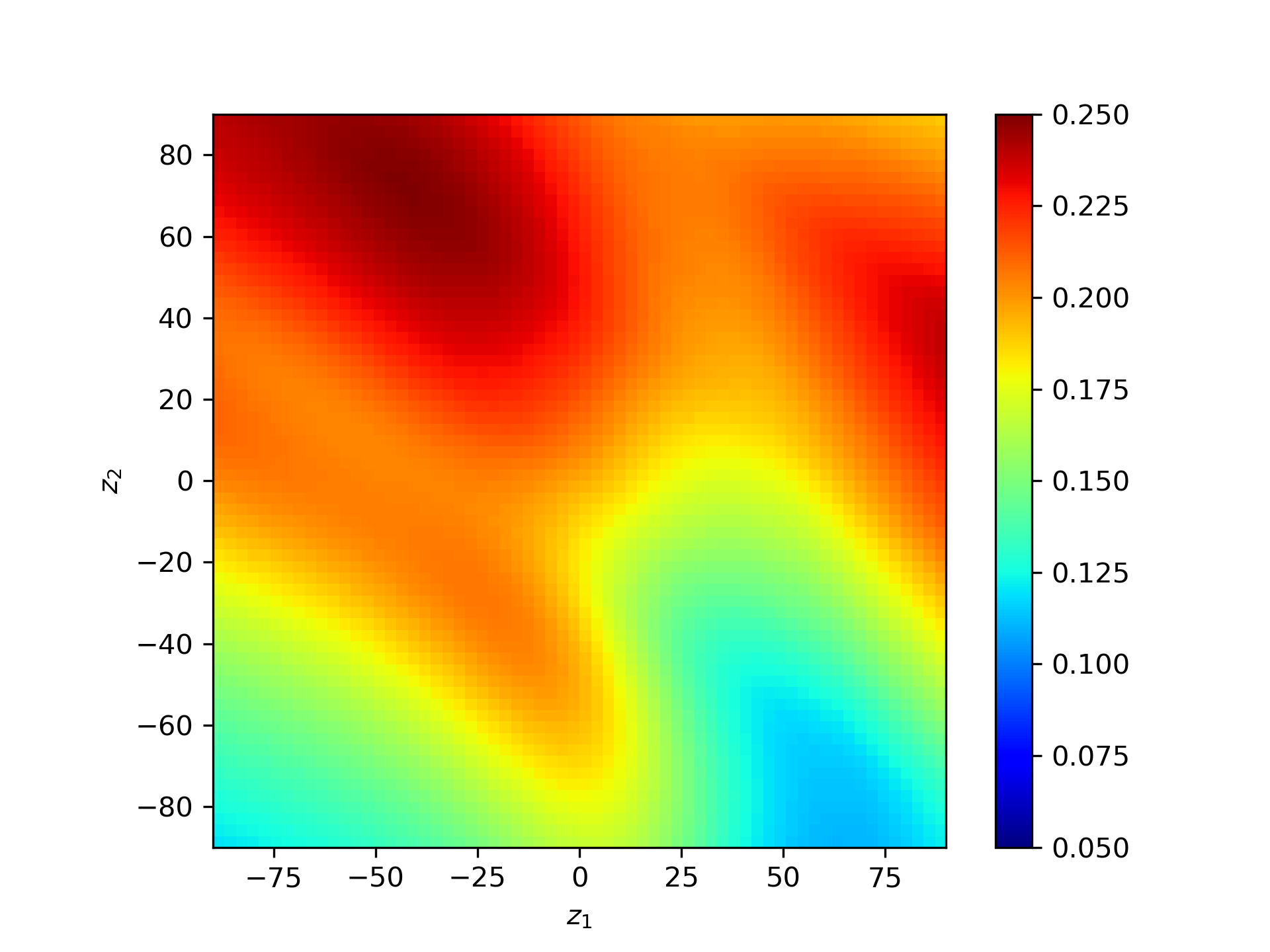}
        \caption{\textit{Basis-derived} method.}
    \end{subfigure}
    \hfill
    \begin{subfigure}[b]{0.32\textwidth}
        \centering
        \includegraphics[width=\textwidth]{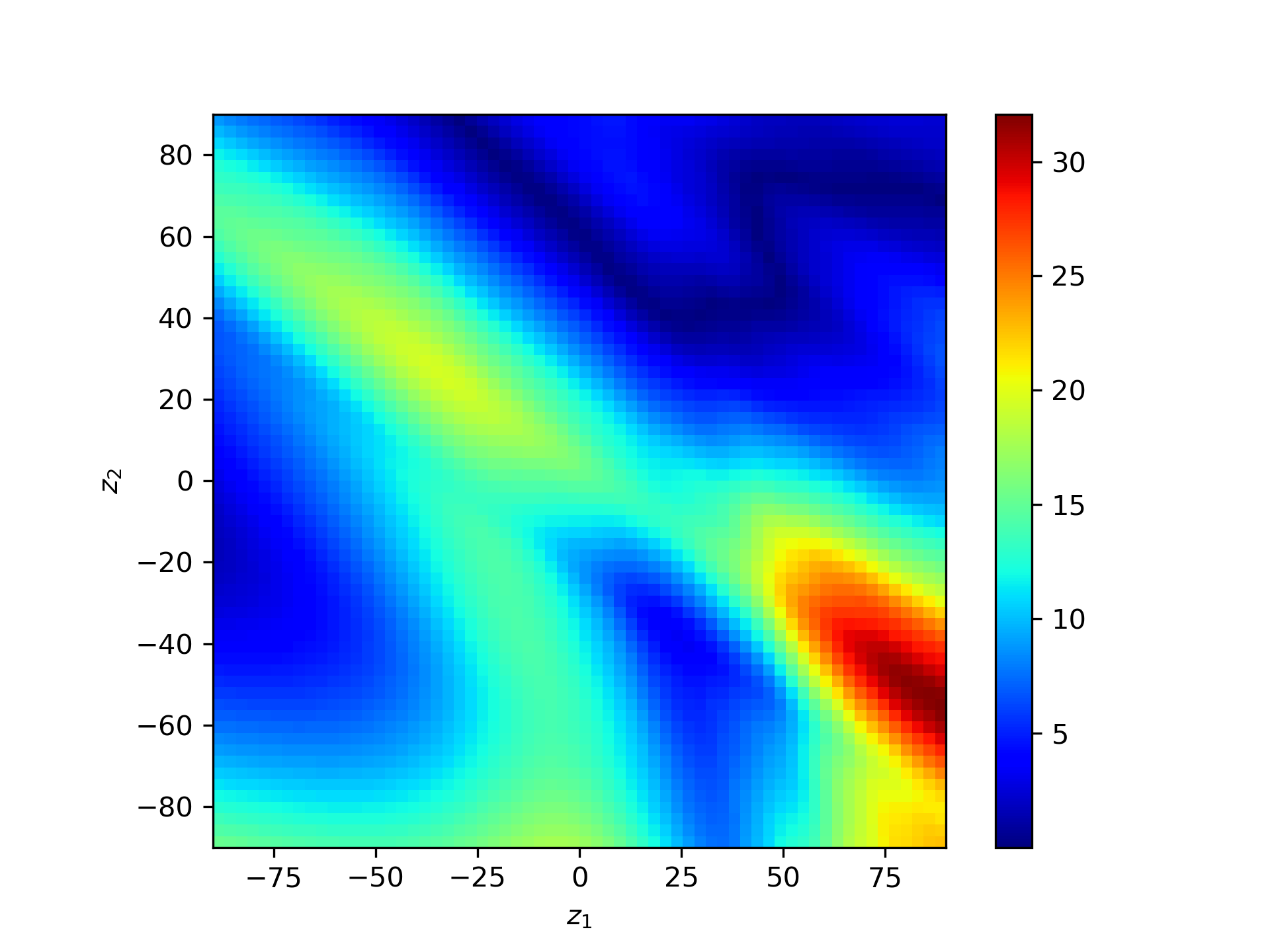}
        \caption{Relative errors.}
    \end{subfigure}
    \caption{Comparison of the first-order SMs for $x_6$ obtained with the direct and \textit{basis-derived} methods. In (c), the color scale values correspond to percentages of the direct method result represented in (a).}
    \label{Fig.first_order_var6}
\end{figure}

\begin{figure}[H]
    \centering
    \begin{subfigure}[b]{0.32\textwidth}
        \centering
        \includegraphics[width=\textwidth]{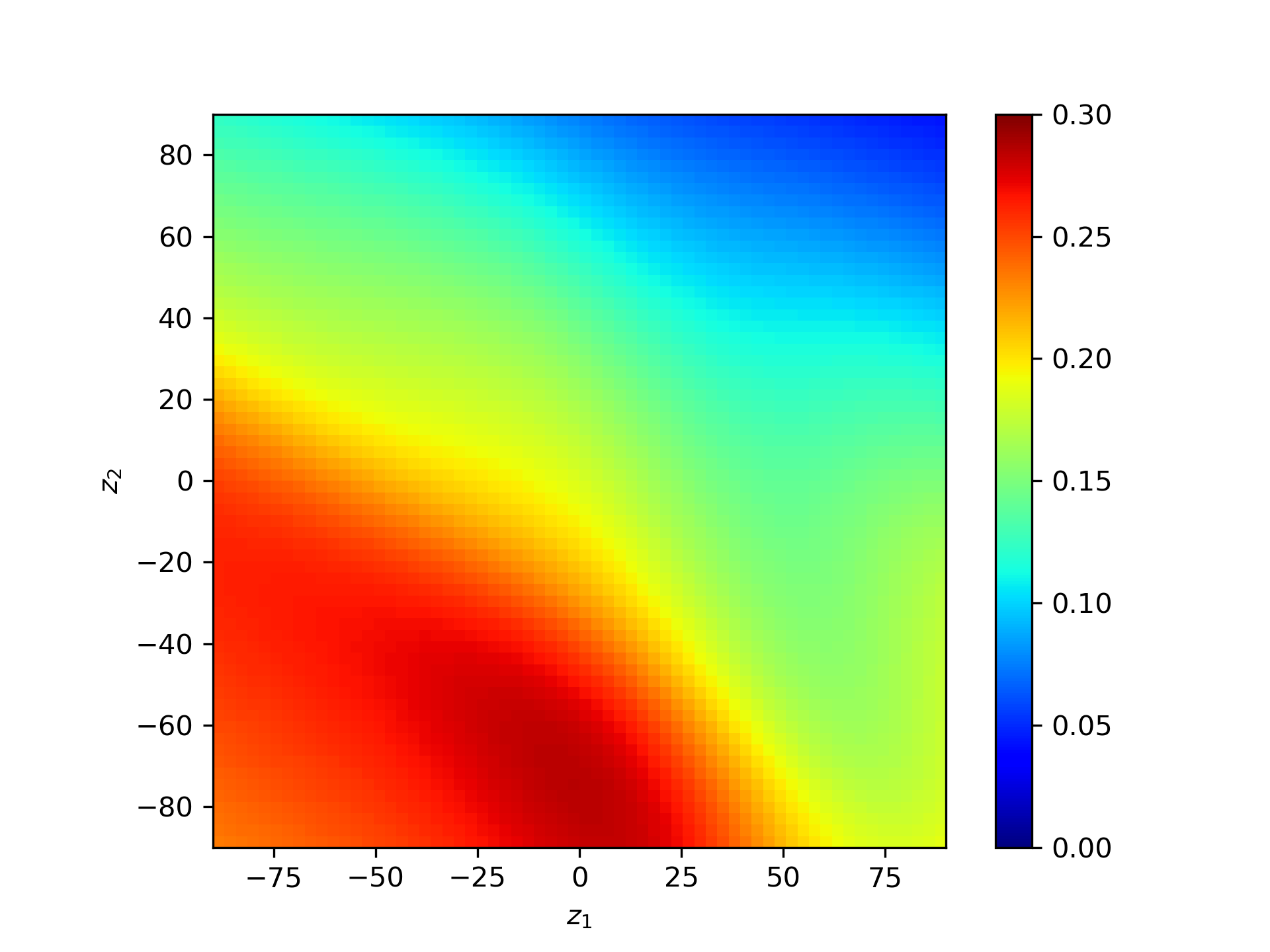}
        \caption{Direct method.}
    \end{subfigure}
    \hfill
    \begin{subfigure}[b]{0.32\textwidth}
        \centering
        \includegraphics[width=\textwidth]{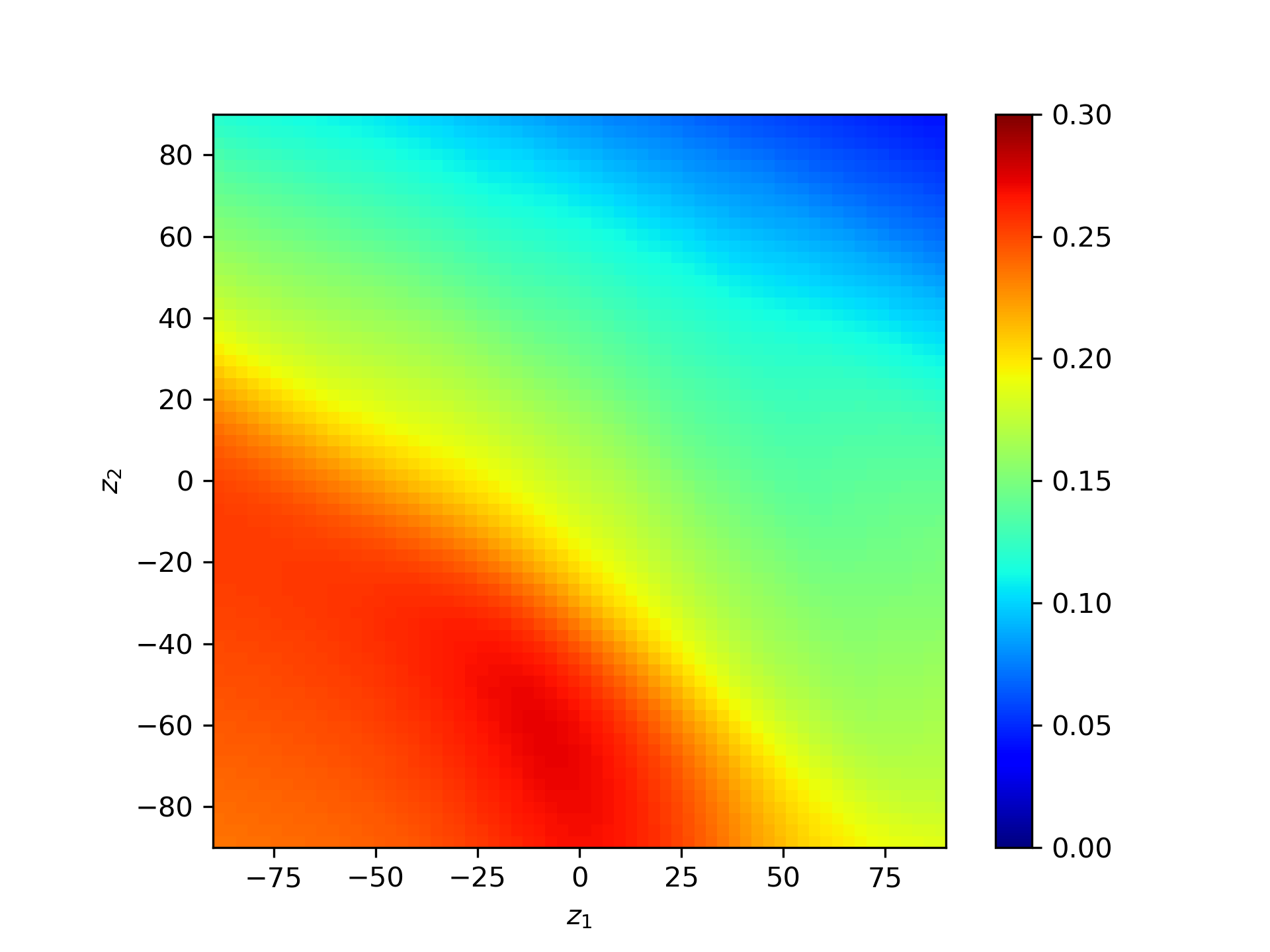}
        \caption{\textit{Basis-derived} method.}
    \end{subfigure}
    \hfill
    \begin{subfigure}[b]{0.32\textwidth}
        \centering
        \includegraphics[width=\textwidth]{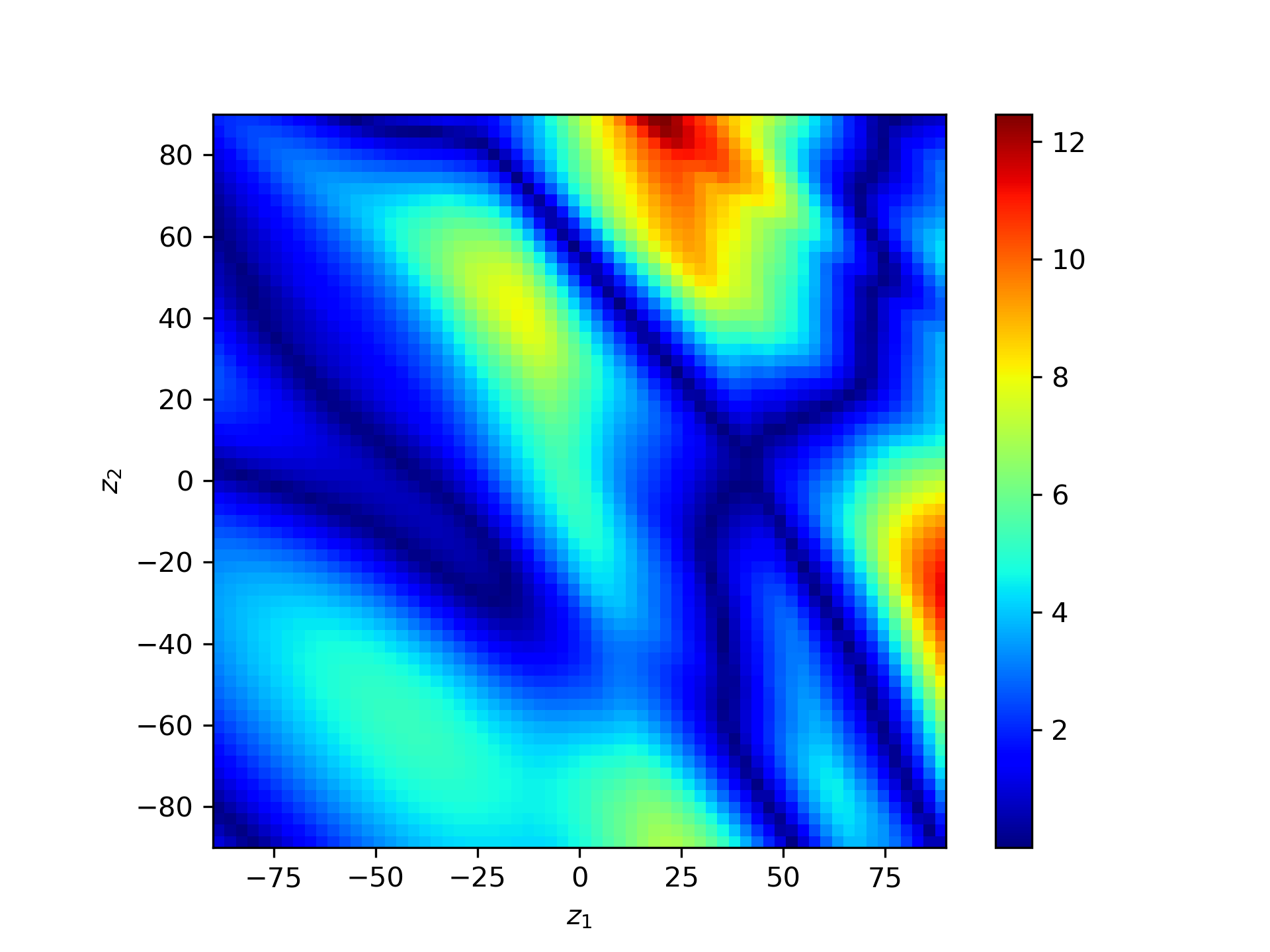}
        \caption{Relative errors.}
    \end{subfigure}
    \caption{Comparison of the total order SMs for $x_4$ obtained with the direct and \textit{basis-derived} methods. In (c), the color scale values correspond to percentages of the direct method result represented in (a).}
    \label{Fig.total_order_var4}
\end{figure}

\begin{figure}[H]
    \centering
    \begin{subfigure}[b]{0.32\textwidth}
        \centering
        \includegraphics[width=\textwidth]{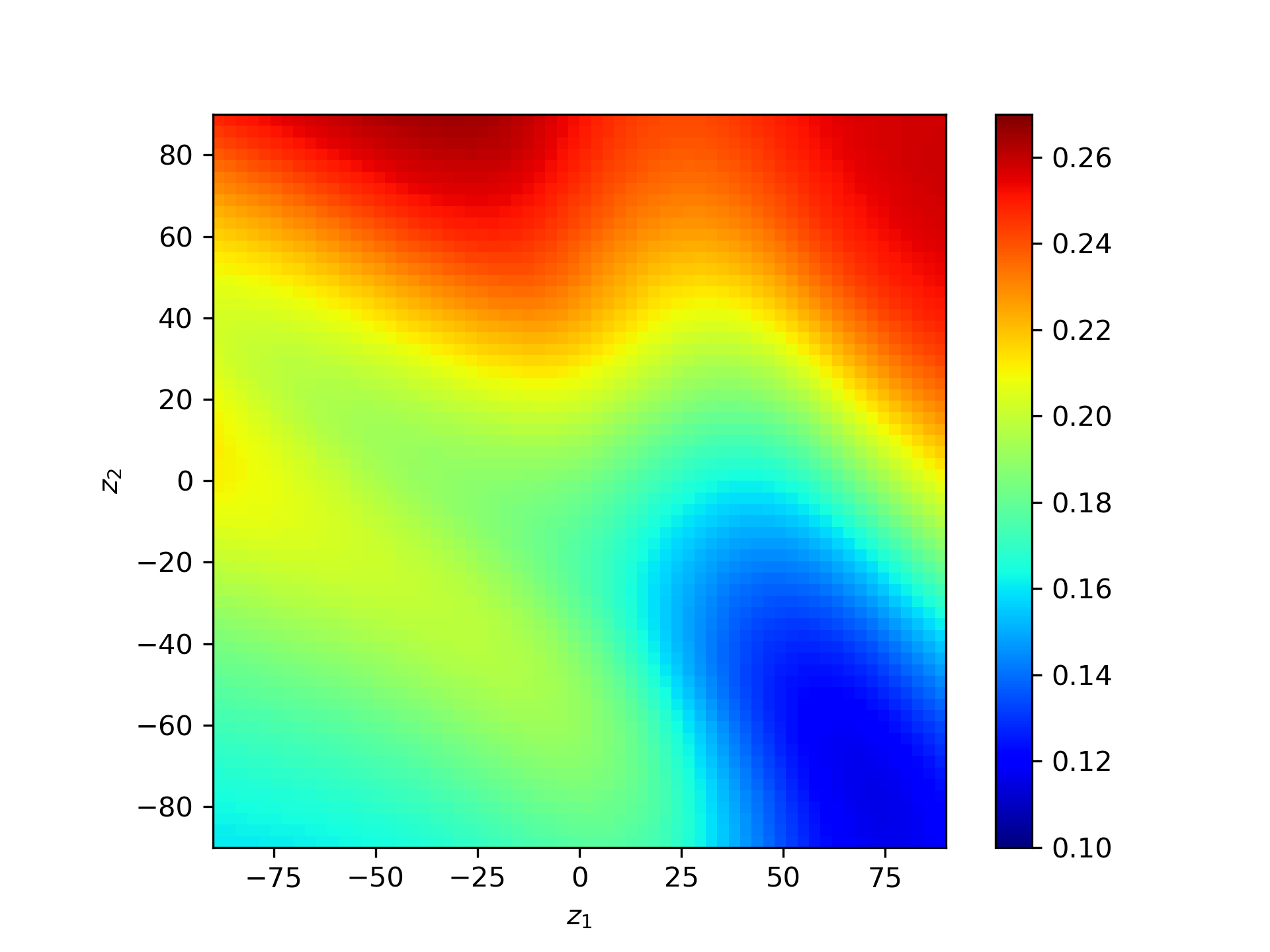}
        \caption{Direct method.}
    \end{subfigure}
    \hfill
    \begin{subfigure}[b]{0.32\textwidth}
        \centering
        \includegraphics[width=\textwidth]{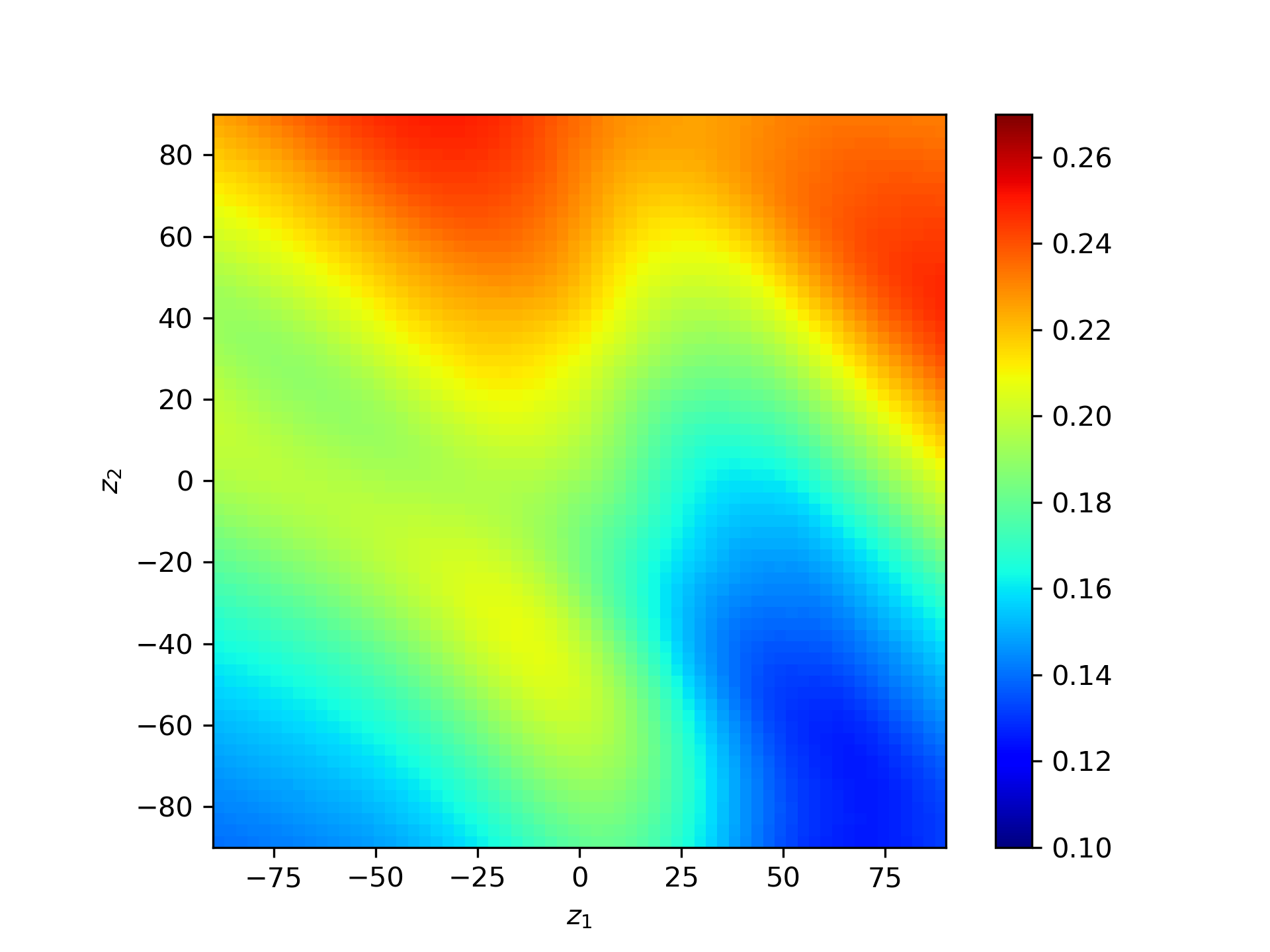}
        \caption{\textit{Basis-derived} method.}
    \end{subfigure}
    \hfill
    \begin{subfigure}[b]{0.32\textwidth}
        \centering
        \includegraphics[width=\textwidth]{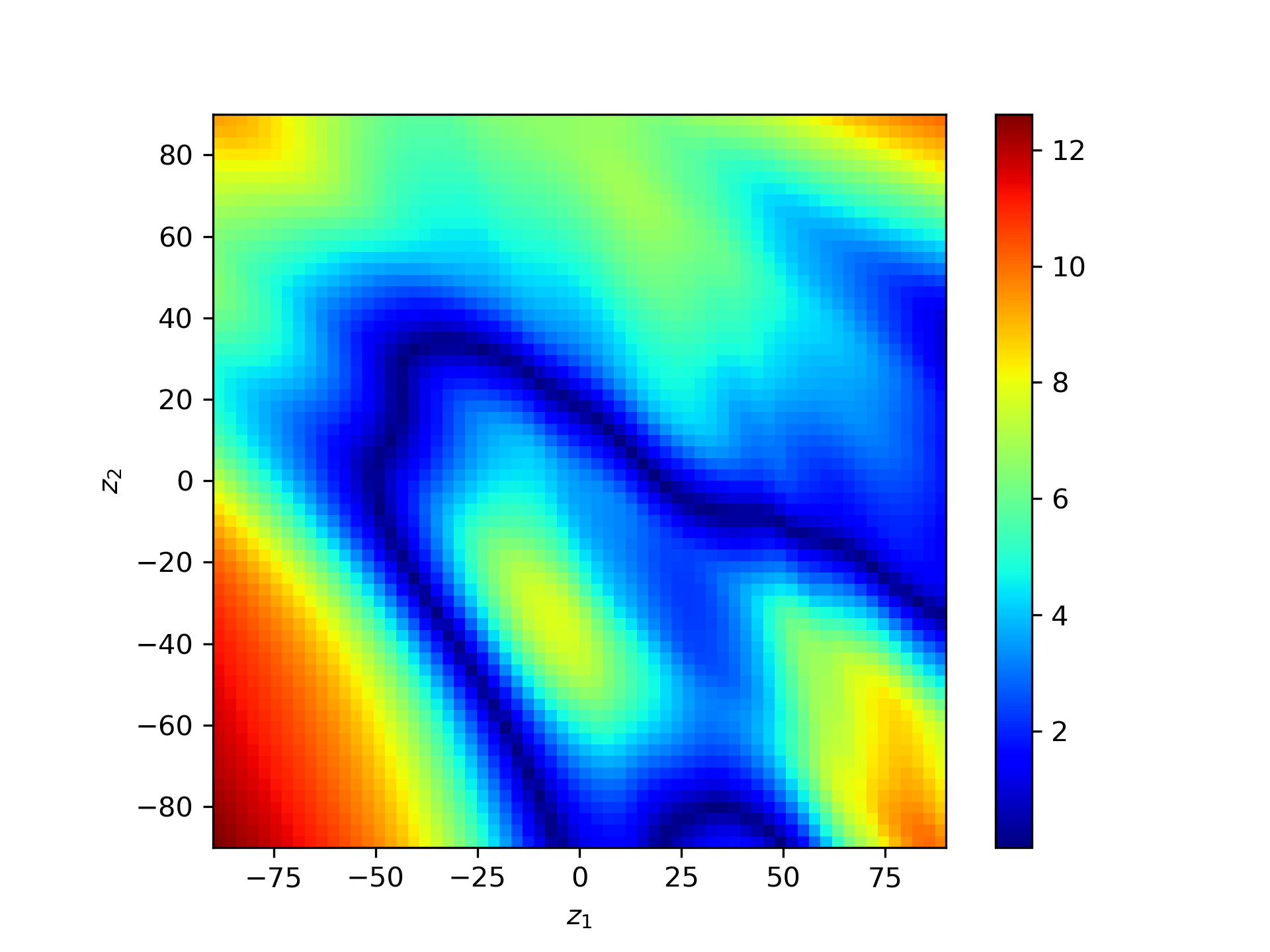}
        \caption{Relative errors.}
    \end{subfigure}
    \caption{Comparison of the total order SMs for $x_8$ obtained with the direct and \textit{basis-derived} methods. In (c), the color scale values correspond to percentages of the direct method result represented in (a).}
    \label{Fig.total_order_var8}
\end{figure}

To summarize the contribution of each input variable over the variance of outputs, the Generalized Sensitivity Index (GSI) can also be calculated (Eq. \ref{eq:gsi}). Figure \ref{Fig.boxplot_campbell} shows the first-order and total indices, where the bootstrap method was also employed. Averaging through all the output domain, it is possible to make some remarks. First, the variables $x_6$ and $x_8$ are the most influential due to their high total index. Second, $x_3$ and $x_5$ are influential just while interacting with other variables; otherwise their main effects are negligible. On the other hand, $x_2$, $x_4$, $x_7$ and $x_8$ are defined almost completely by their main effect. Finally, $x_1$ is not influential, neither in its main effect nor in its interactions with other variables. The same results were obtained by \citet{Perrin2021}.

\begin{figure}[H]
    \centering
    \hfill
        \centering
        \includegraphics[width=\textwidth]{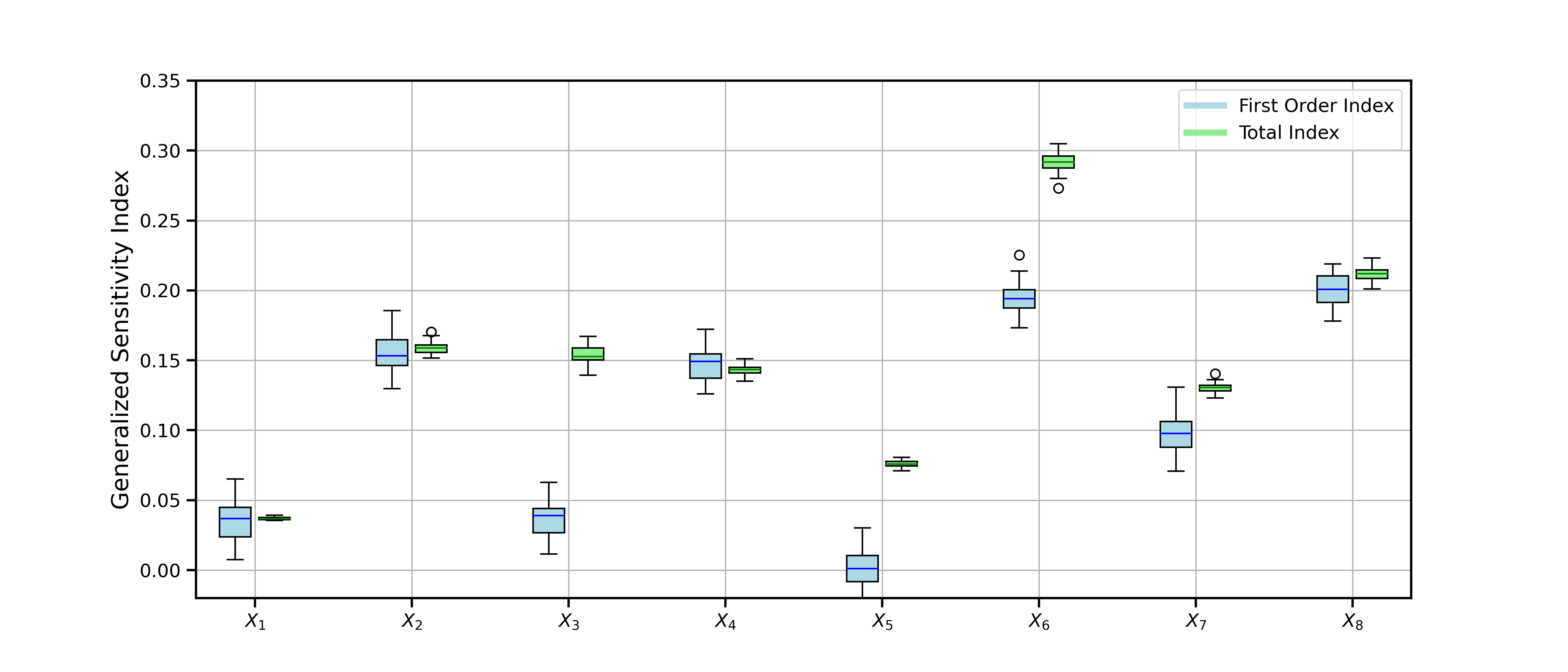}
    \caption{Generalized Sensitivity Indices of input variables.}
    \label{Fig.boxplot_campbell}
\end{figure}

Using Proposition~\ref{prop:ratio_costs} considering the given parameters ($\numberPickFreezeSamples = 5\,000$, $\nbasis = 7$ and $\nOutputDimensions = 4096$), the theoretical gain is $\mathrm{\costDW}/\mathrm{\costBD} = 330$, i.e. the \textit{basis-derived} approach is approximately $330$ times faster than the \textit{dimension-wise} approach. The practical computational gain was measured and is approximately $30$ times faster for the \textit{basis-derived} approach. Two factors can explain this 10-fold difference, especially those related to the implementation of the code: the specificity of the coding language and the hardware. About coding, some operations were executed in native Python and others in functions based on lower-level and more efficient languages (such as C and FORTRAN), which can degrade performance. For example, the ``decoding'' procedure in the \textit{dimension-wise} approach is supposed to be, in theory, the most dominant term in $\mathrm{\costDW}$ (cost of $4\nbasis\numberPickFreezeSamples\nOutputDimensions$, compared to the pick-freeze cost of $8 \numberPickFreezeSamples \nOutputDimensions$). However, using a highly efficient scalar product function (numpy.dot), it corresponds to half of the practical $\mathrm{\costDW}$ at maximum. Since the other terms cannot perform as efficiently, some disparities of practice with theory may occur. About the hardware, 
the simulations were performed in two setups: 1) AMD Ryzen 7 6800H 3.2 GHz with 16 logical processors; and 2) Apple M3 with 8 cores (used in all simulations). The first setup provided a cost ratio
2 times greater (approximately $60$) than the second setup, although the second setup was faster in absolute computational time. It shows that the hardware can also affect the practical computational costs.

\subsection{Idealized gradual dam-break flows of non-Newtonian fluids}
\label{Subsec.dam_break}

This section presents the GSA study of the idealized case of gradual dam-break flow of non-Newtonian fluids. An idealized gradual dam-break flow consists of a known volume of material inside a reservoir delimited by the walls and a gate, which is lifted with a finite velocity and the material flows downstream a horizontal plane or channel (Fig. \ref{Fig.schematic}) \citep{Matson2007, Ancey2009, Liu2016}. 

\begin{figure}[H] 
\centering
\includegraphics[width=1.\textwidth]{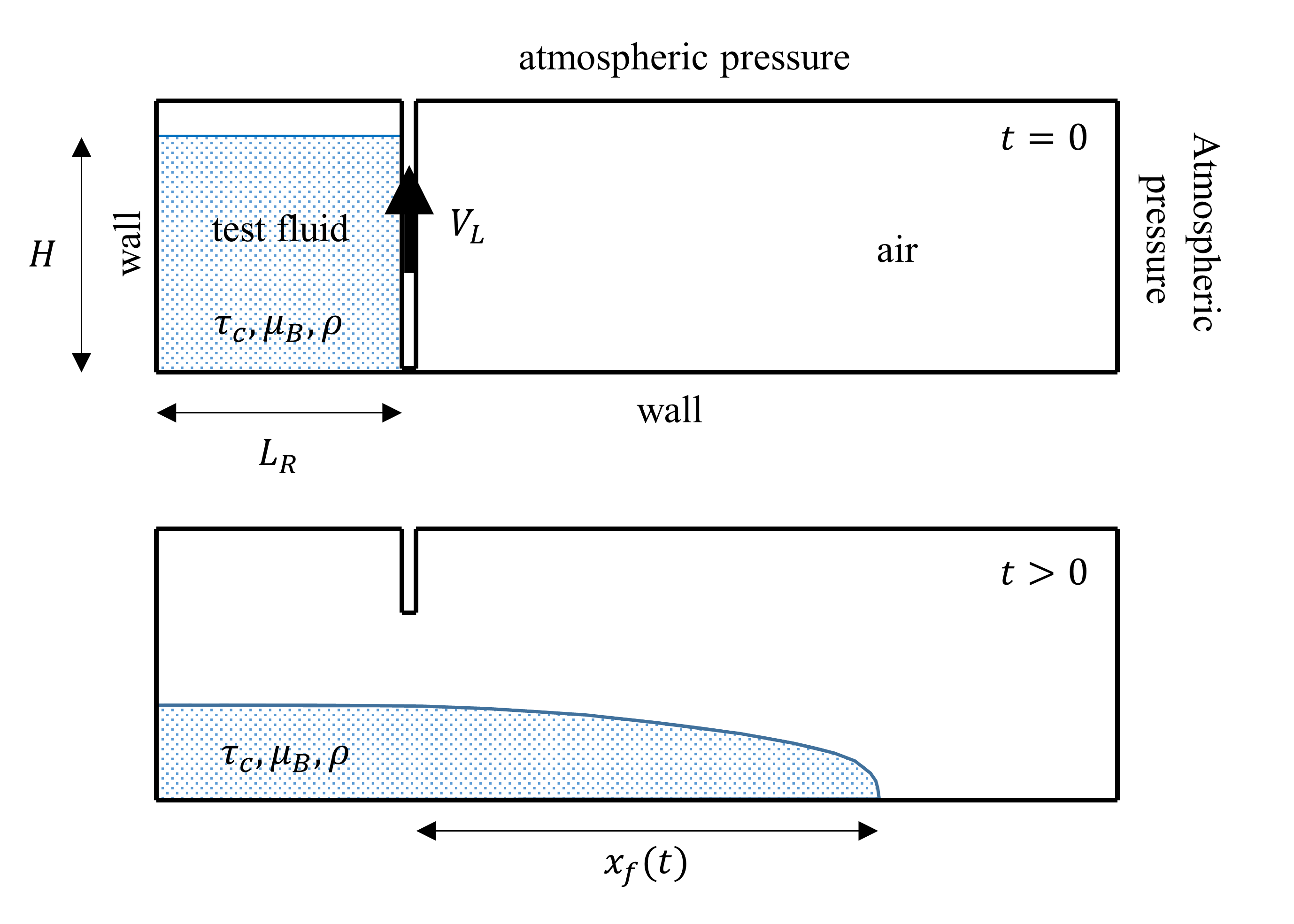}
\caption{Schematic of the idealized gradual dam-break case with boundary conditions and input variables.}\label{Fig.schematic}
\end{figure}

This application is relevant to many areas of engineering in the context of evaluating rheological properties and physical characteristics of non-Newtonian materials, such as fresh concrete, mud, gels, food, etc. We refer to \citet{Chhabra2011, Irgens2014, Balmforth2014} for further details of non-Newtonian behavior of fluids. While rheological properties are usually assessed by equipment called ``rheometers'', alternative techniques have been developed and studied, providing acceptable results with less costs \citep{Clayton2003, Pereira2021, Sao2021, Gao2015flume, Balmforth2007}. For each technique, particular geometric characteristics (e.g. maximum distance and/or maximum fluid height) are related to rheological properties through empirical correlations or analytical models \citep{Pashias1996, Pereira2022}. In this context, the case presented herein is important since it consists on the physical basis of relevant alternative rheometry techniques, such as the ``slump test'' \citep{Clayton2003, Pereira2021, Pereira2022}, ``mini-slump test'' \citep{Pashias1996, Gao2015mini} and the ``Bostwick consistometer'' \citep{Balmforth2007, Minussi2012}. 

We aim to apply the contribution of this work to study the influence of input variables (initial height, fluid properties and lifting dynamics) of the problem over a chosen quantity of interest: the position of the wavefront over time. In that way, we can obtain sensitivity indices along the entire time series, allowing us to have insights about the phenomena considering lifting dynamics, fluid characteristics and initial geometry.

\subsubsection{Data acquisition}

Data obtained in this work was produced by 2D numerical simulations using ANSYS Fluent r20.0, which employs the Finite Volume Method to discretize and solve the system of equations of continuity and momentum from fluid mechanics. Figure \ref{Fig.schematic} shows the computational domain and boundary conditions. Structured quadrangular meshes with $\Delta x = 5 \times 10^{-4}$ $m$ of resolution were used and the Layering technique was applied to address the moving boundary (lifting gate) \citep{Gao2015mini, Pereira2022}. The multiphasic model Volume of Fluid (VoF) was used to model different phases by introducing the variable Volume Fraction $\alpha_{(\cdot)}$ ($\alpha_1=0$ for air, $\alpha_2=1$ for test fluid, where $\alpha_1 + \alpha_2 = 1$) and to track the free surface of the test fluid ($0 < \alpha_2 < 1$) \citep{Gopala2008}. To model the non-Newtonian fluid rheology, the regularized Bingham model was employed \citep{Bird1983, Balmforth2014, Gao2015mini}, whose constitutive equation in apparent viscosity $\eta$ form is given by:

\begin{equation}
    \eta = 
    \begin{cases}
        \frac{\tau_c}{\dot{\gamma}_c} \left( 2 - \frac{\dot{\gamma}}{\dot{\gamma}_c}\right) + \mu_B & \text{if } \dot{\gamma} \leq \dot{\gamma}_c \\
        \frac{\tau_c}{\dot{\gamma}} + \mu_B & \text{if } \dot{\gamma} > \dot{\gamma}_c
    \end{cases}
\end{equation}

where $\tau_c$ is the yield stress, $\mu_B$ is the plastic viscosity, $\dot{\gamma}$ is the shear rate and $\dot{\gamma}_c$ is the critical shear rate. We considered $\dot{\gamma}_c = 10^{-3}$ $s^{-1}$, which is sufficient to regularize the viscosity \citep{Gao2015mini, Pereira2022}. 

Five input variables were considered: rheological properties $\tau_c$ and $\mu_B$, density $\rho$, initial height $H$ and lifting velocity $V_L$. The variables were considered uniformly distributed as shown in Tab. \ref{tab.prob_distributions}. The minimum and maximum values of rheological properties and density are based on the literature \citep{Obrien1988, Major1992, Sosio2007, Blight2009, Sarsby2000}. For lifting dynamics and initial height, we explored very low/high velocities and low/high aspect ratios given the fixed unitary volume of $0.26$ $m^2$ and the prototypical dimensions of the domain.  

\begin{table}[h!]
    \centering
    \begin{tabular}{|c|c|}
        \hline
        Variable & Distribution \\
        \hline
        $\tau_c$ & $\mathcal{U}(0.1 , 200)$ \text{Pa} \\
        \hline
        $\mu_B$ & $\mathcal{U}(0.01 , 15)$ \text{Pa.s} \\
        \hline
        $\rho$ & $\mathcal{U}(1000 , 2650)$ \text{Kg.m\textsuperscript{-3}} \\
        \hline
        $V_L$ & $\mathcal{U}(0.01 , 1)$ \text{m.s\textsuperscript{-1}} \\
        \hline
        $H$ & $\mathcal{U}(0.2 , 1)$ \text{m} \\
        \hline
    \end{tabular}
    \caption{Input variables and their distributions.}
    \label{tab.prob_distributions}
\end{table}

A pure Monte-Carlo approach was used to generate an initial DoE of size $130$. Then, the DoE was enriched with a Latin Hypercube sample, producing a final sample of size $226$. The sample size was limited by the computational cost of each simulation. The simulations were performed and the quantity of interest evaluated was the wavefront position (runout) over time, or $x_f \times t$. Each output is composed of a sequence of $x_f$ data points irregularly spread over time. 
To make this sequence regular, a linear interpolation was applied to each output. Finally, $226$ curves were produced, formed by regular sequences of $5\,000$ interpolated points over time. Therefore, in this case, the number of output dimensions is $\nOutputDimensions=5\,000$.

\subsubsection{Results and discussion}

The same procedure applied in Subsec. \ref{Subsec.campbell2d} was employed herein: outputs of the DoE were expanded through PCA ($\nbasis = 10$ accounting for more then $99.9 \%$ of the total variance) and the resulting basis coefficients were metamodeled using GPR. To evaluate the accuracy of the predictions provided by metamodels, we defined $11$ DoEs ($100 \leq n_T \leq 200$, in steps of $10$) and calculated the respective $Q^2$ metrics (Eq. \ref{eq:q_squared}), as Figure \ref{Fig.dambreak_q2} shows. The curves refer to the bootstrap median (20 bootstrap replicates) and only the most accurate set of metamodels ($n_T = 200$) presents the 1st and 3rd quartiles area, for visualization reason. It is observable that the $Q^2$ score increases by increasing the size of the DoE, until approximately $n_T = 170$. For $n_T > 170$, the $Q^2$ score curves remain without significant variation. These score curves present $Q^2 \approx 0.8$ in early instants of time because of large variation of data in these instants; then, they increase to $Q^2 \approx 0.92$, showing these values during the entire remaining time series. We consider it as an acceptable prediction score and we chose $n_T = 200$ to carry out the analysis.

\begin{figure}[H]
\centering 
\includegraphics[width=1.1\textwidth]{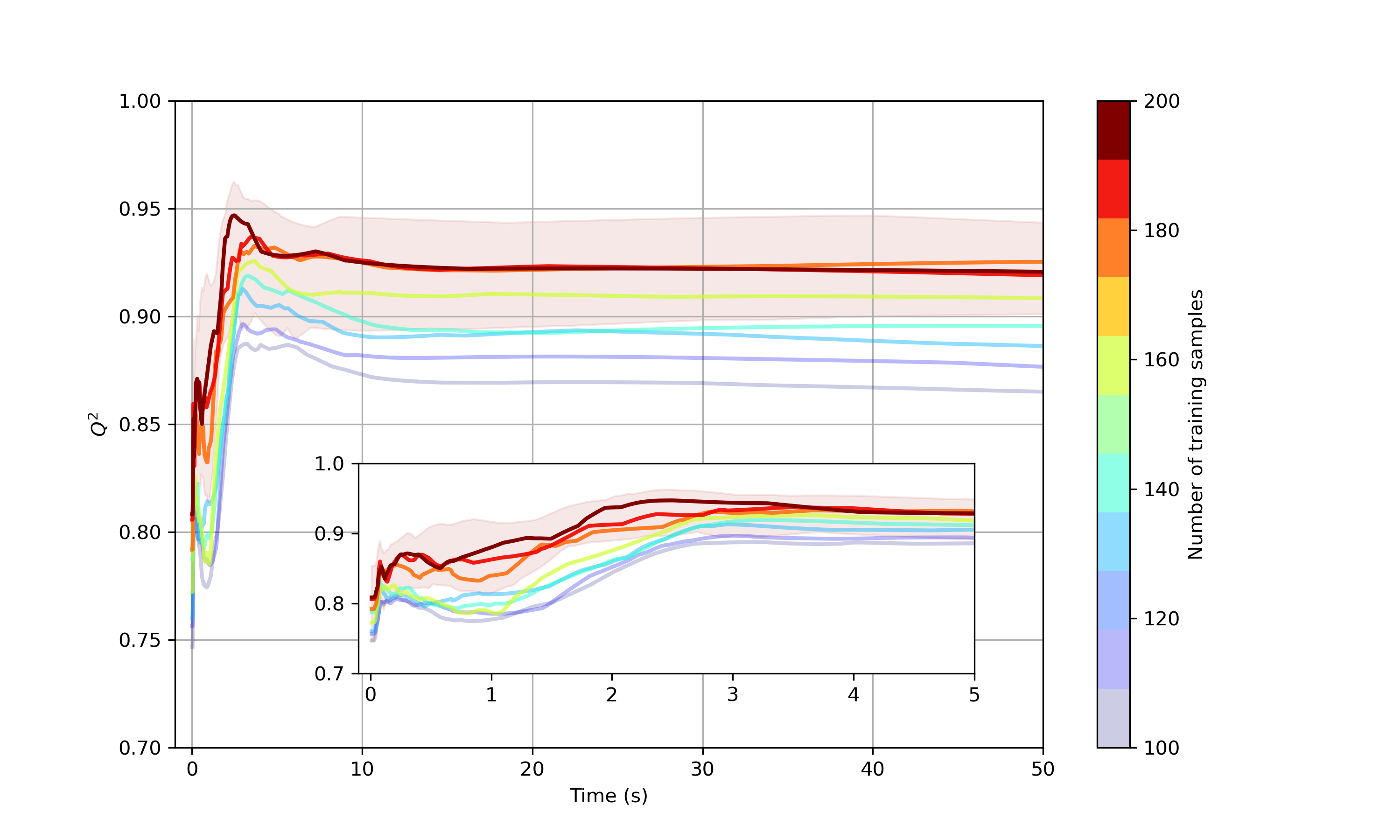}

\caption{Accuracy of the prediction according to the $Q^2$ metric. The curves correspond to the bootstrap median and the shaded region is the area between the 1st and 3rd quartiles of the bootstrapping set (20 bootstrap replicates). A zoomed plot from times $t=0$ and $t=5$ $s$ is presented for better visualization of initial time instants. For visualization reasons, the shaded region is shown just for the DoE of size $200$.}\label{Fig.dambreak_q2}
\end{figure}

For the GSA, PF samples of size $5\,000$ were generated and the Sobol' indices were computed using Definition \ref{Def.Pick_freeze_vector}, i.e. using the \textit{basis-derived} approach. The bootstrap method was used to quantify the uncertainty related to the PF sample size. For this, $50$ replicates were used, allowing us to obtain the median and standard deviation of each SM. Figure \ref{Fig.dambreak_sobol} shows the entire time series of first-order (dashed lines) and total Sobol' indices (continuous lines) for each input variable over the runout $x_f$. It is globally observed that the plastic viscosity $\mu_B$ is most influential variable in most of the time series, followed by the yield stress $\tau_c$, initial height $H$, density $\rho$ and finally the lifting velocity $V_L$. It shows that the rheological properties of the fluid are, in fact, important for the runout, which shows the link between the dam-break flow and estimation of rheological properties.

\begin{figure}[H]
\centering
\includegraphics[width=1.1\textwidth]{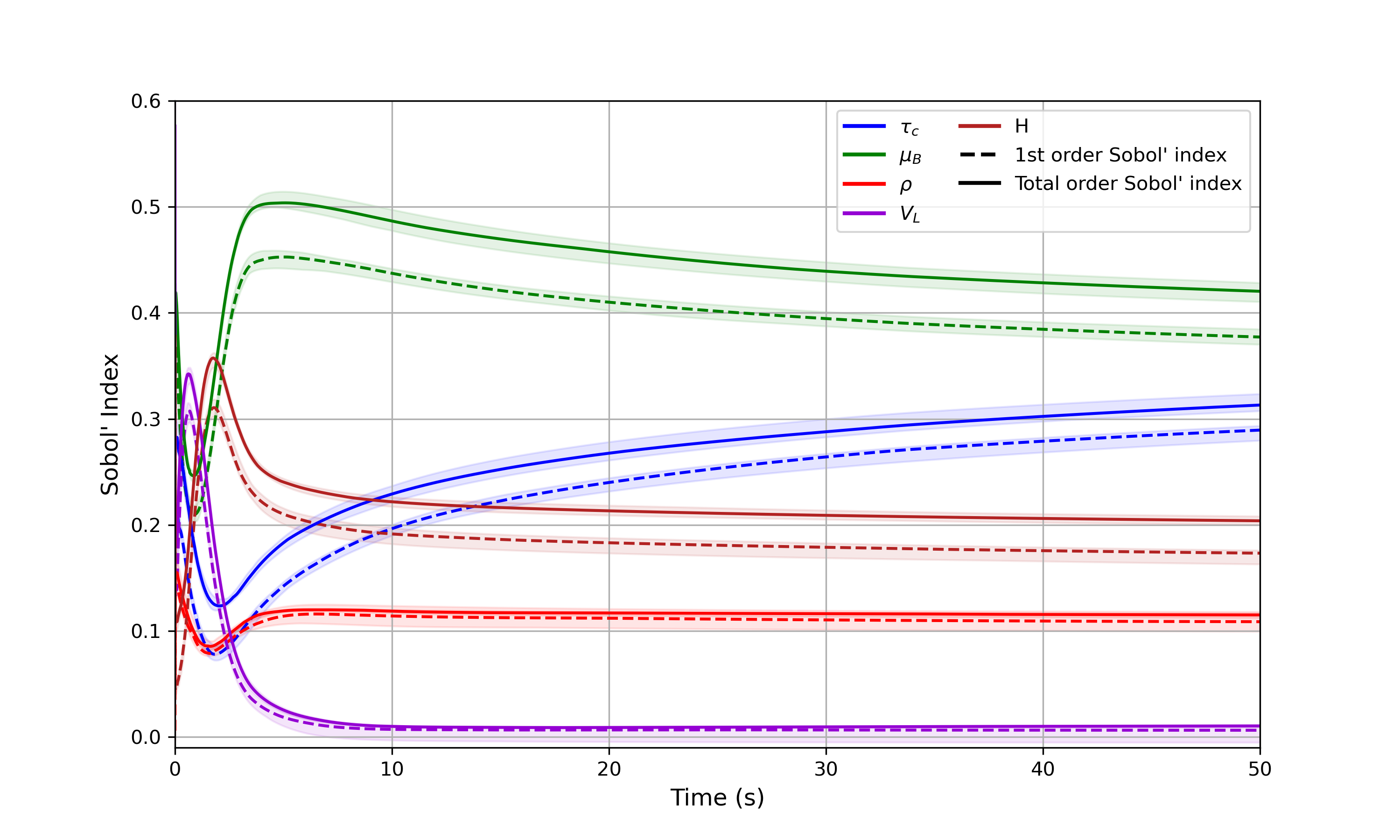}

\caption{Total and first-order SMs of all input variables. Total Sobol' indices are represented by solid lines, whereas first-order Sobol' indices are represented by dashed lines. The estimation error is assessed by boostrap with $50$ replicates: lines correspond to the mean and shaded regions to the first and third bootstrap quartiles.}

\label{Fig.dambreak_sobol}
\end{figure}

However, it is clear that the sensitivity of each input variable depends on time: in the beginning of the time series, $V_L$ and $H$ are more influential than in the end of the time series; conversely, $\tau_c$ becomes increasingly more influential later in the time series. Evaluating $\tau_c$ considering $t \rightarrow \infty$, its increasing influence is coherent with non-Newtonian fluid mechanics, since in steady regime the material remains static and the yield stress is perfectly counterbalanced by hydrostatic pressure \citep{Pashias1996, Gao2015flume, Pereira2021}. It means that the equilibrium position depends strongly on the yield stress, which is corroborated by our analysis.

To evaluate the initial times, Fig. \ref{Fig.dambreak_sobol_zoomed} shows the Sobol' indices for $0 < t < 5$ $s$. The behavior of $V_L$ and $H$ makes physical sense. For $V_L$, a fast lifting of the gate releases more material, which ultimately affects the runout position $x_f$ in the beginning of the time series; in later instants, this effect is counter-balanced by viscous forces ($\mu_B$ and $\tau_c$) and $V_L$ influence disappears. For $H$, a higher value means more gravitational potential energy, also affecting $x_f$, but its influence in later instants is only partially counter-balanced by viscous forces.

\begin{figure}[H]
\centering
\includegraphics[width=1.\textwidth]{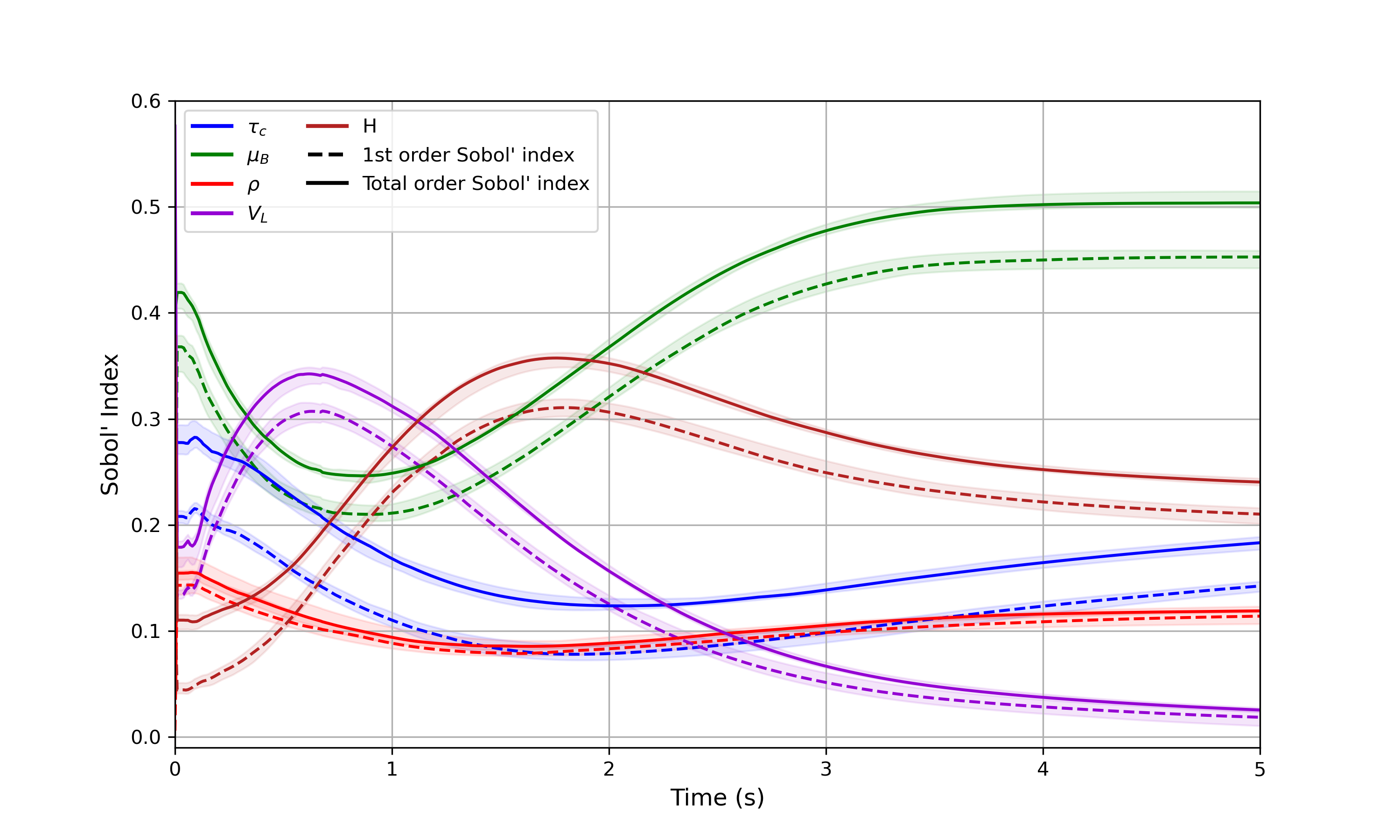}
\caption{A zoomed plot of Fig. \ref{Fig.dambreak_sobol} on the initial time instants of the time series.}
\label{Fig.dambreak_sobol_zoomed}
\end{figure}

The global influence of each input variable can be described by the Generalized Sensitivity Indices (GSI) using Eq. \ref{eq:gsi}, as Fig. \ref{Fig.dambreak_boxplot} shows. The GSI values confirm that, indeed, $\mu_B$ and $\tau_c$ are the most influential parameters, followed by $H$, $\rho$ and $V_L$. The slight difference between GSI and total GSI indicate that interactions between variables are not strong; the same can be concluded by checking the first-order and total SMs from Figs. \ref{Fig.dambreak_sobol} and \ref{Fig.dambreak_sobol_zoomed}. We should note that the GSI analysis is not able to interpret locally (or instantaneously) the influence of each input variable, although it possesses a clear advantage which is producing a concise index to describe globally the influences.

\begin{figure}[H]
\centering
\includegraphics[width=1.\textwidth]{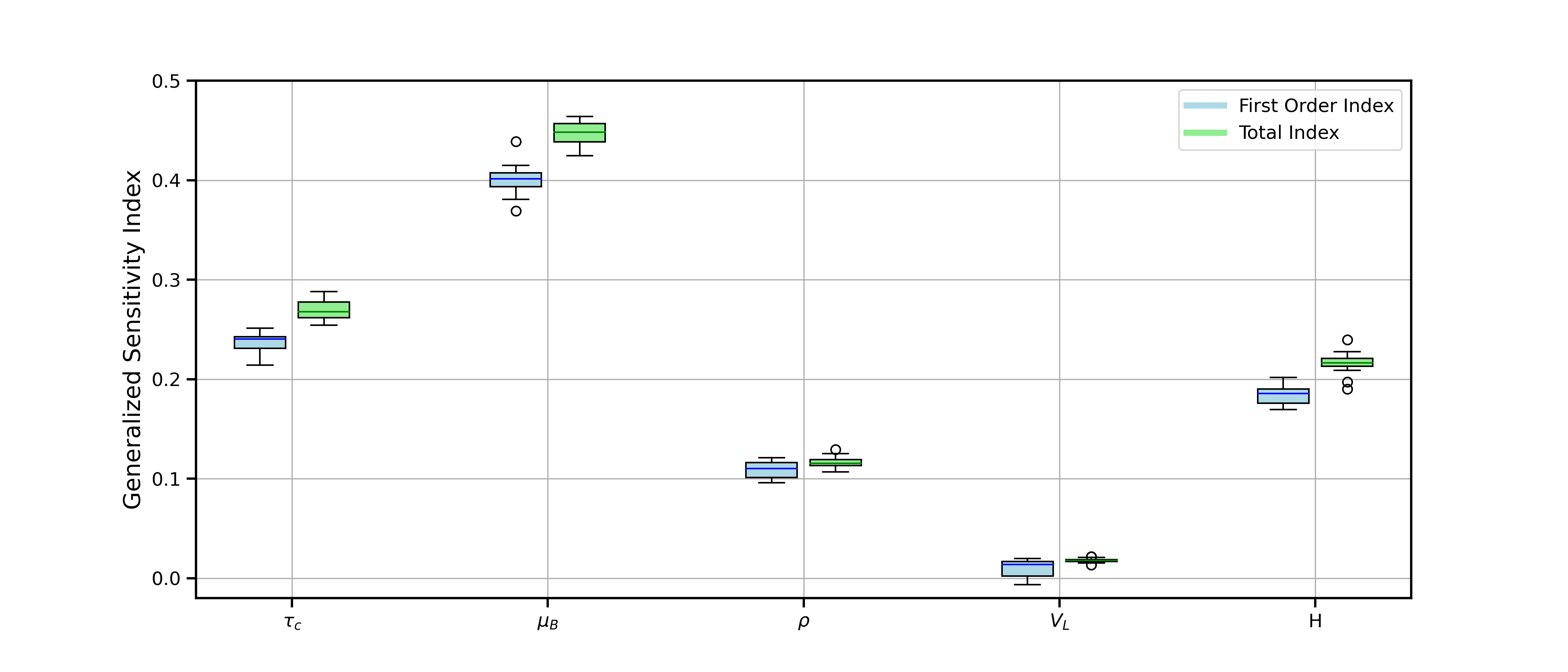}
\caption{Comparison between first-order Generalized Sensitivity Indices (GSI) and Total GSI for each input variable.}
\label{Fig.dambreak_boxplot}
\end{figure}

The theoretical computational gain for the parameters given in this case is $\mathrm{\costDW}/\mathrm{\costBD} = 252$, whereas the practical computational gain is approximately $35$ faster for the \textit{basis-derived} approach. The same observations made for the computational costs of Campbell2D function can be applied here.

\section{Conclusion and future works}
\label{sec.conclusion}

In this work, we employed a fast PF estimator of Sobol' indices in the context of high-dimensional output models with basis-expanded data, aiming to compute \textit{Sensitivity Maps} (SMs), called herein \textit{basis-derived} approach. Statistical estimation of Sobol' indices was explored and formulas for closed Sobol' indices were given, with an immediate extension to Sobol' indices of any order and total Sobol' indices. This estimator uses directly the basis coefficients and basis components, instead of using the original form of data output (\textit{dimension-wise} approach). It was proven that the \textit{basis-derived} approach involves much fewer operations in comparison with the \textit{dimension-wise} approach, hence ``fast'', which allows to calculate both SM and their associated bootstrap confidence bounds in a reasonable time.
This advantage also impacts the calculation of other quantities, such as the Generalized Sensitivity Index (GSI), which uses the covariance matrices calculated by the method. Furthermore, this estimator is general and can be applied with any basis expansion (e.g. wavelets, b-splines) and metamodeling techniques (e.g. neural networks, Polynomial Chaos Expansion).  

We applied the estimator in two applications, using the Principal Component Analysis (PCA) and Gaussian Process Regression (GPR) techniques: the analytical Campbell2D function and the idealized gradual dam-break flow of non-Newtonian fluid. Overall, the coupled procedure of dimension-reduction with metamodeling was capable of reducing the dimensionality of outputs and of producing accurate results with less computational resources. Moreover, the procedure worked with limited DoE size, in the order of $200$, which is particularly useful for high-cost simulations. The procedure allowed us to apply the \textit{basis-derived} PF estimator using a PF sample of sufficient size to make reasonable analyses. 

For the first application, the \textit{basis-derived} PF method showed to be capable of recovering the theoretical SMs with an acceptable accuracy, while performing less operations. For the second application, the method provided the continuous time series of Sobol' indices, which allowed us to better understand the phenomena by checking the sensitivity of each input variable over the output data. It was found that the influence of input variables change considerably in function of time, which reflects the transient nature of the phenomena. Non-newtonian characteristics of the fluid (plastic viscosity and yield stress) represented the major source of contribution over the total variance and the lifting dynamics of the gate only contributed in initial times of the flow.

\section*{Acknowledgments} 

We would like to acknowledge the Fundação de Amparo à Pesquisa do Estado de São Paulo (FAPESP) for the Ph.D. scholarship funding (processes 2022/05184-1 and 2023/08472-0) of the first author. Part of the research has been done in the frame of the Chair PILearnWater, part of the AI Cluster ANITI, funded by the French National Research Agency.

 \bibliography{cas-refs}
 \bibliographystyle{elsarticle-harv} 

\end{document}